\documentclass[onefignum,onetabnum]{siamart171218}

\usepackage{bbm}
\newtheorem{remark}[theorem]{Remark}
\newtheorem{assumption}[theorem]{Assumption}

\usepackage{mathtools}

\usepackage{subfigure}
\usepackage{lipsum}
\usepackage{amsfonts}
\usepackage{graphicx}
\usepackage{epstopdf}
\usepackage{algorithmic}
\ifpdf
  \DeclareGraphicsExtensions{.eps,.pdf,.png,.jpg}
\else
  \DeclareGraphicsExtensions{.eps}
\fi

\newcommand{\TheTitle}{Generalizing parallel replica dynamics} 
\newcommand{\TheAuthors}{D. Aristoff}

\headers{\TheTitle}{\TheAuthors}

\title{{Generalizing parallel replica dynamics: trajectory fragments, asynchronous computing, and PDMPs}\thanks{This work was supported by the
National Science Foundation 
via the awards
NSF-DMS-1522398 and 
NSF-DMS-1818726.}}

\author{
  David Aristoff\thanks{Colorado State University}
    (\email{aristoff@math.colostate.edu})
  }

\usepackage{amsopn}

\begin{document}

\maketitle

\begin{abstract} 
We study the Parallel Replica 
Dynamics in a general 
setting. We introduce 
a trajectory fragment framework that can 
be used to design and prove 
consistency of Parallel 
Replica algorithms 
for generic Markov 
processes. We use our 
framework to formulate 
a novel condition that 
guarantees an 
asynchronous algorithm 
is consistent. Exploiting 
this condition and 
our trajectory fragment 
framework, we
present new 
synchronous and 
asynchronous Parallel Replica 
algorithms for piecewise 
deterministic Markov processes.
\end{abstract}



\section{Introduction}

Many problems in applied sciences 
require the sampling of 
complex probability distributions. 
In computational chemistry -- 
which is the main setting 
of this article -- such 
distributions can arise from 
stochastic models of molecular 
dynamics~\cite{lelievre2010free} or chemical reaction networks~\cite{anderson2011continuous}, 
while obstacles to efficient 
sampling include high dimensionality 
and metastability, the latter being 
the tendency to become stuck in certain 
subsets of state space~\cite{lelievre2013two}. 
Some attempts to surmount 
these difficulties have been based on importance sampling and 
stratification~\cite{shirts2008statistically,thiede2016eigenvector,torrie1977nonphysical,vardi1985empirical}, interacting particles~\cite{de2001sequential,del2004feynman,del2014particle,del2005genealogical}, coarse graining and preconditioning~\cite{aristoff2016analysis,lyman2006resolution,tempkin2014using}, accelerated dynamics~\cite{lelievre2015accelerated,so2000temperature,voter1997hyperdynamics,voter1998parallel} 
and nonreversibility~\cite{duncan2016variance,goodman2010ensemble,lelievre2013optimal,rey2015irreversible,wu2014attaining}.

This article concerns the 
Parallel Replica Dynamics
(ParRep)~\cite{voter1998parallel}, 
an accelerated 
dynamics method designed to 
overcome metastability. ParRep has 
two distinct advantages over many other 
enhanced sampling methods. First, 
it computes correct {\em dynamical}~\cite{aristoff2014parallel,le2012mathematical,voter1997hyperdynamics} as well as 
{\em stationary} or {\em equilibrium}~\cite{aristoff2015parallel,wang2018stationary} quantities associated with a stochastic process. 
And second, ParRep is very general: it only 
requires mild assumptions on the underlying process. 
Indeed, though originally 
intended for Langevin dynamics~\cite{le2012mathematical,voter1997hyperdynamics,voter2012accelerated}, straightforward 
extensions of ParRep 
to discrete and 
continuous time 
Markov chains have appeared in~\cite{aristoff2015parallel,aristoff2014parallel,wang2018stationary}.

The goal of this article 
is as follows. First, we introduce  a new mathematical framework that may be used to design, and 
prove consistency of, ParRep algorithms for Markov processes satisfying a few mild assumptions. Second, we use our framework 
to obtain valuable insights 
about 
asynchronous computing. 
In particular, we present specific, novel 
conditions that ensure an asynchronous ParRep algorithm 
is consistent.
Lastly, we construct 
ParRep algorithms 
for piecewise deterministic 
Markov processes (PDMPs) in both the 
synchronous and asynchronous setting, leaning on our new 
framework to demonstrate their consistency. 
Asynchronous ParRep 
algorithms must be carefully designed, since as we show below, 
inaccuracies can arise
when the speed of computing paths of the underlying process is 
coupled to the process itself.

PDMPs are emerging as an useful 
tool in fields as 
diverse as applied probability~\cite{monmarche2014piecewise}, computational chemistry~\cite{alfonsi2005adaptive,bierkens2017piecewise,harland2017event,kang2013separation,michel2014generalized,nishikawa2016event,rudnicki2017piecewise,salis2005accurate,winkelmann2017hybrid}, machine 
learning~\cite{bouchard2017bouncy,pakman2016stochastic,wu2017generalized}, 
and big data~\cite{bierkens2016zig}. 
As indicated by the name, PDMPs
move along deterministic paths in between 
random jump times. In the context of chemical reaction networks,
PDMPs called {\em hybrid models} can be obtained 
by approximating
fast reactions by a deterministic 
flow, and representing slow reactions with an 
appropriate Poisson process~\cite{kang2013separation,winkelmann2017hybrid}. The 
resulting PDMPs can be metastable~\cite{bressloff2017stochastic,bressloff2018stochastic,bressloff2013metastability,bressloff2014path,newby2015bistable,newby2012isolating,newby2014spontaneous,newby2013breakdown,newby2011asymptotic}, 
making direct simulation
unattractive. Several 
PDMP-based algorithms have 
also been proposed for sampling from distributions known up to normalization -- like the Boltzmann distribution or the posterior distribution in Bayesian 
analysis -- including
Event chain Monte Carlo~\cite{harland2017event,michel2014generalized}, the Zig-Zag process~\cite{bierkens2016zig}, 
and the Bouncy Particle Sampler~\cite{bouchard2017bouncy,monmarche2014piecewise}. 
Below, we give a general argument 
suggesting these PDMPs also 
become metastable under certain 
conditions.

This article is organized as follows.  Section~\ref{sec:notation} 
defines notation that 
we use throughout. In Section~\ref{sec:metastability}, 
we formally define metastability 
in terms of {quasistationary distributions}.  
We describe ParRep in more detail, and explain what we mean 
by a consistent ParRep 
algorithm, in Section~\ref{sec:parrep}. 
In Section~\ref{sec:framework} 
we outline a general mathematical 
framework for ParRep, and in 
Section~\ref{sec:asynch} we study 
synchronous and 
asynchronous computing. 
Section~\ref{sec:PDMPs} 
serves as a brief introduction 
to PDMPs, while 
Section~\ref{sec:parrep_pdmp} 
outlines several ParRep 
algorithms for PDMPs that 
are based on our framework 
from Section~\ref{sec:framework}. 
A numerical example is in 
Section~\ref{sec:numerics}.
All proofs are in Section~\ref{sec:proofs}.

\section{Notation}\label{sec:notation}

Throughout,
$X(t)_{t \ge 0}$ is 
a time 
homogeneous Markov process, {either discrete or continuous 
in time}, with values in a standard Borel state space; 
$U$ is a subset 
of state space; and $g$ 
is a real-valued 
function defined on 
state space. 
Without 
explicit mention we assume
all sets are measurable and all functions are 
bounded and measurable.
We write 
$X(t)$ to refer 
to the process $X(t)_{t \ge 0}$ at 
time $t$. We 
denote various expectations and probabilities by
${\mathbb E}$ and ${\mathbb P}$, 
with the precise meaning being 
clear from context. 
We write ${\mathcal L}$  
for the probability law of a 
random object, with ${\mathcal L}$ 
above an equals sign indicating
equality in law. We say a random 
object is a {\em copy} of another random 
object if it has the same law as that object. 
When we say a collection of random 
objects is independent we mean 
these objects are {\em mutually} independent unless otherwise specified.
We define $a\wedge b = \min\{a,b\}$ and 
$a \vee b =\max\{a,b\}$, 
and write $\lfloor s\rfloor$ 
for the greatest integer less than or equal to $s$.

\section{Metastability}\label{sec:metastability}
Informally, $U$ is a metastable 
set for $X(t)_{t \ge 0}$ if 
$X(t)_{t \ge 0}$ tends to reach 
a local equilibrium in $U$ much faster than it escapes from $U$. Local equilibrium 
can be understood in terms of {\em quasistationary distributions} (QSDs):

\begin{definition}
Fix a subset
$U$ of state space, and consider
$$T = \inf\{t\ge 0:X(t) \notin U\},$$
the 
first time $X(t)_{t \ge 0}$ escapes~$U$.
A QSD $\rho$ of $X(t)_{t \ge 0}$ in 
$U$ satisfies $\rho(U) = 1$ and
\begin{equation}\label{QSDdef}
\rho(A) = {\mathbb P}(X(t) \in A|{\mathcal L}(X(0)) = \rho,\,T>t)
\end{equation}
for every $t \ge 0$ and $A \subseteq U$. 
\end{definition}

Note that $\rho$ is supported in $U$. Equation~\eqref{QSDdef} states 
that if $X(0)$ is distributed as $\rho$ and $X(t)_{t \ge 0}$
does not escape from $U$ by time $t$, then $X(t)$ is distributed 
as $\rho$.
Throughout, we will assume the QSD of $X(t)_{t \ge 0}$ in 
$U$ exists, is unique, and is the long time distribution of $X(t)$ conditioned to never escape $U$. 
That is, we assume that for any initial distribution of $X(0)$ supported in $U$,
\begin{equation}\label{QSDconv}
\rho(A) = \lim_{t \to \infty} {\mathbb P}(X(t) \in A|X(s) \in U \text{ for }s \in [0,t]) \quad \forall \,A \subseteq U.
\end{equation}
The QSD $\rho$ can then be sampled as follows: choose a time  $T_{corr}^\rho(U)$ 
for relaxation to $\rho$. 
Start $X(t)_{t \ge 0}$ in $U$, 
and if it escapes from
$U$ before time $t = T_{corr}^\rho(U)$, restart it in 
$U$. Repeat this until 
a trajectory of $X(t)_{t \ge 0}$ 
remains in $U$ for a consecutive time 
interval of length $T_{corr}^\rho(U)$. 
This trajectory's terminal position is 
then a sample of $\rho$. For more 
details on the QSD see for 
instance~\cite{collet2012quasi}. 
For conditions ensuring existence of and 
convergence to the QSD for general 
Markov processes, see~\cite{champagnat2016exponential,champagnat2017general,collet2012quasi}.

A more formal definition of metastability is: a set $U$ is metastable for $X(t)_{t \ge 0}$
if the time scale to reach $\rho$ is small 
compared to the mean time to escape from $U$ starting at $\rho$. In 
some cases these times can be written
in terms of the eigenvalues 
of the adjoint, $L^*$, of the generator $L$ of $X(t)_{t \ge 0}$,
with absorbing boundary conditions on the complement of~$U$.
See~\cite{le2012mathematical} 
and~\cite{wang2018stationary} 
for the corresponding 
spectral analysis for
overdamped Langevin dynamics and finite state space discrete and continuous time Markov chains, and~\cite{binder2015generalized} 
for an application of these 
ideas to choosing $T_{corr}^\rho(U)$.

\section{Parallel replica dynamics}\label{sec:parrep}

ParRep can boost the efficiency of 
simulating metastable processes~\cite{aristoff2015parallel,aristoff2014parallel,le2012mathematical,voter1998parallel,voter2012accelerated,wang2018stationary}.
Currently, implementations have 
been proposed only for Langevin 
or overdamped Langevin dynamics~\cite{le2012mathematical} 
and discrete or continuous time Markov chains~\cite{aristoff2015parallel,aristoff2014parallel,wang2018stationary}. However, the 
generality of ParRep allows 
for extensions to any metastable
time homogeneous strong Markov 
process with c\`adl\`ag paths, in cases where the QSD exists and
metastable sets can be identified. 
We make this precise in 
the next section.

ParRep algorithms are based 
on two basic steps: 
\vskip5pt
\begin{itemize}
\item A step in which $X(t)_{t \ge 0}$ is allowed to reach the QSD in some metastable set $U$, 
using direct or serial simulation -- called the {\em decorrelation step};
\item A step generating an escape event from $U$, starting from the QSD, using parallel simulation -- called the {\em parallel step}.
\end{itemize}
\vskip5pt

By {\em escape event} 
we mean the random pair $(T,X(T))$, where 
$T$ is the time for 
$X(t)_{t \ge 0}$ 
to escape from $U$ when 
$X(0)$ is distributed as the QSD 
in $U$, and $X(T)$ is the 
corresponding
escape point. The parallel 
step efficiently computes an escape 
event starting from the QSD 
via a sort of time parallelization.

The decorrelation step, as 
it uses only serial 
simulation, is exact. 
By {\em exact} we mean there 
is {\em zero error} -- except for 
the inevitable error 
in simulating $X(t)_{t\ge 0}$ 
arising from numerical discretizations, 
which we will ignore.
Our analysis 
will therefore focus on the 
parallel step. 

The parallel step is sometimes divided 
into two subrouties: first, 
a routine that generates 
independent samples 
of the QSD in $U$ -- called {\em dephasing} -- and 
second, a routine that 
uses copies of $X(t)_{t \ge 0}$ 
starting from these QSD samples to generate 
an {escape event} of $X(t)_{t \ge 0}$ 
from $U$. Below, we will mostly omit discussion of the dephasing 
routine, and we will not discuss the error 
associated with imperfect convergence 
to the QSD in the dephasing and decorrelation 
steps, as these points have been previously studied in~\cite{binder2015generalized,le2012mathematical,simpson2013numerical,wang2018stationary}.

We say the parallel step of a ParRep algorithm is 
{\em consistent} when:
\vskip5pt
\begin{itemize}
\item The parallel step generates  
escape events from each metastable set $U$ with the correct probability law -- see Theorem~\ref{theorem_dyn} below; 
\item The parallel step produces 
correct mean contributions to 
time averages in each metastable set $U$ -- see Theorem~\ref{theorem_avg} below.
\end{itemize}
\vskip5pt

By {\em correct} we mean {exact} 
provided the QSD sampling has zero error. 
A consistent ParRep algorithm 
defines a 
coarse dynamics, that is,
a dynamics that is correct 
on the quotient space obtained 
by considering each metastable 
set as a single point~\cite{aristoff2015parallel,aristoff2014parallel,wang2018stationary}. A 
consistent ParRep algorithm also 
defines stationary averages that are 
correct for functions 
defined on the original 
uncoarsened state 
space~\cite{aristoff2015parallel,wang2018stationary}. ParRep produces only a coarse dynamics because the parallel step does not resolve the
exact behavior of $X(t)_{t \ge 0}$. The parallel step is faithful enough to $X(t)_{t\ge 0}$, however,
to produce correct stationary averages on the original uncoarsened state space~\cite{aristoff2015parallel}.

Previous analyses of ParRep 
have relied on the structure 
of $X(t)_{t \ge 0}$ and the particular algorithms 
studied~\cite{aristoff2015parallel,aristoff2014parallel,le2012mathematical,wang2018stationary}. We introduce a new
framework below that allows us 
to study the consistency of 
any ParRep algorithm.  Our 
analysis is inspired by 
{\em ParSplice}, a recent implementation 
of ParRep employing asynchronous 
computing~\cite{perez2017parsplice,perez2018long,swinburne2018self}. 
Our framework provides 
explicit conditions that 
ensure an asynchronous ParRep algorithm 
is consistent. In particular, it shows a certain class 
of asynchronous ParRep 
algorithms is consistent 
provided the wall-clock time 
to simulate a step of 
$X(t)_{t \ge 0}$
is not coupled to its 
position in state space; see 
Section~\ref{sec:asynch} 
below for precise statements.

\section{Trajectory fragments}\label{sec:framework}

We now formalize 
conditions which lead 
to consistency of ParRep. 
Our arguments are based on what 
we call {\em trajectory 
fragments}. The 
fragments are copies of the underlying process 
satisfying the dependency conditions of 
Assumption~\ref{A1} below. 
In practice, the trajectory fragments 
may be computed asynchronously 
in parallel. We discuss this in the 
next section.

\begin{assumption}\label{A1}
Let $X(t)_{t \ge 0}$ have 
c{\`a}dl{\`a}g paths and 
the strong Markov property. 
Assume $X(t)_{t \ge 0}$ has a QSD $\rho$
in an open set $U$ and that $T = \inf\{t \ge 0:X(t) \notin U\}$ is finite almost surely. 
Let
$\,(X_{m}(t)_{t\ge 0},T_m)_{m \ge 1}$ be copies of $(X(t)_{t \ge 0},T)$ such that:
\begin{align}
&\text{conditional 
on }X_m(0), X_m(t)_{t \ge 0} \text{ is 
independent of }(X_k(t)_{t \ge 0})_{1 \le k < m}; \label{assump1}\\
&{\mathcal L}(X_{m}(0)|T_{m-1}>t_{m-1},\ldots,T_1> t_1) = \rho \text{ for } m \ge 2, \quad {\mathcal L}(X_1(0)) = \rho,\label{assump2}
\end{align}
where $t_m>0$ are deterministic times satisfying $\sum_{m=1}^{\infty} t_m = \infty$.
\end{assumption}

\begin{algorithm}[b!]
\caption{A general parallel step in $U$}  
Let
Assumption~\ref{A1} hold and adopt the notation 
therein.
\vskip5pt
{\bf 1.} Define
$L = \inf\{m \ge 1:T_m \le t_m\}$.

{\bf 2.} In the discrete time case, set 
\begin{equation*}
g_{par} = {\mathbb E}\left(\sum_{m=1}^{L} \sum_{t=0}^{T_m \wedge t_m-1} g(X_m(t))\right) 
\end{equation*}
while in the continuous time case, set
\begin{equation*}
g_{par} = {\mathbb E}\left(\sum_{m=1}^{L} \int_0^{T_m \wedge t_m} g(X_m(t))\,dt\right).
\end{equation*}

{\bf 3.} Let $T_{par} =  t_1 + \ldots + t_{L-1}+T_L$ and $X_{par} = X_L(T_L)$.
\vskip5pt

Once $g_{par}$, $T_{par}$ and $X_{par}$ can be computed, the parallel step is complete. This 
parallel step is consistent in the sense of Theorem~\ref{theorem_dyn} and 
Theorem~\ref{theorem_avg} below. \label{alg_gen}
\end{algorithm}

The $X_m(t)_{0 \le t \le t_m}$ 
from Assumption~\ref{A1} 
are the {\em trajectory 
fragments}. We will 
refer to $X_m(t)_{t > t_m}$ 
as a fragment's {\em 
irrelevant future}. The 
reason for this choice 
of words is that the 
output of a general 
parallel step, described 
in Algorithm~\ref{alg_gen}, is the same no matter how the $X_m(t)_{t \ge 0}$ 
are defined for times $t>t_m$.

Algorithm~\ref{alg_gen}  
outlines a general parallel step. As discussed above, this parallel step 
can be combined with a decorrelation step 
to compute a coarse dynamics or 
a time average of a function $g$.
The idea behind 
Algorithm~\ref{alg_gen} 
is simple -- we  
imagine concatenating 
fragments 
whose starting points 
and terminal points are distributed 
as the QSD $\rho$, thus  
obtaining an artificial
long trajectory. See Figure~\ref{fig_intuition} below.
One must be careful, however, 
in treating
dependencies of the fragments. 
The dependencies described 
in Assumption~\ref{A1} lead 
to a consistent Algorithm~\ref{alg_gen} 
in the sense of Theorems~\ref{theorem_dyn}-\ref{theorem_avg}. 
More general dependencies 
can violate consistency, 
as we will discuss in 
the next section.

Our next two results demonstrate 
consistency of Algorithm~\ref{alg_gen} under 
the conditions in 
Assumption~\ref{A1}. 
Theorem~\ref{theorem_dyn} 
states that Algorithm~\ref{alg_gen} produces
the correct escape law 
from $U$ starting at the 
QSD in $U$, 
while Theorem~\ref{theorem_avg} says that Algorithm~\ref{alg_gen} produces 
the correct mean contribution to time averages.

\begin{theorem}\label{theorem_dyn}
Suppose that Assumption~\ref{A1} holds. Let $X(t)_{t \ge 0}$ be such that ${\mathcal L}(X(0)) = \rho$, and set $T = \inf\{t>0:X(t) \notin U\}$. Then in Algorithm~\ref{alg_gen},  $$(T_{par},X_{par}) \stackrel{\mathcal L}{=} (T,X(T)).$$
\end{theorem}

\begin{theorem}\label{theorem_avg}
Suppose that Assumption~\ref{A1} holds. Let $X(t)_{t \ge 0}$ be such that ${\mathcal L}(X(0)) = \rho$, and set $T = \inf\{t>0:X(t) \notin U\}$. Then in Algorithm~\ref{alg_gen}, in the discrete time case, 
\begin{equation*}
g_{par} := {\mathbb E}\left(\sum_{m=1}^{L} \sum_{t=0}^{T_m \wedge t_m-1} g(X_m(t))\right) = {\mathbb E}\left(\sum_{t=0}^{T-1}g(X(t))\right),
\end{equation*}
while in the continuous time case, 
\begin{equation*}
g_{par} := {\mathbb E}\left(\sum_{m=1}^{L} \int_0^{T_m \wedge t_m} g(X_m(t))\,dt\right) = {\mathbb E}\left(\int_0^{T}g(X(t))\,dt\right).
\end{equation*}
\end{theorem}

Recall that the gain in ParRep is 
from parallel computations in the parallel 
step. In our trajectory 
fragment framework,
the basic idea is that the work 
to compute the fragments $X_m(t)_{0 \le t \le t_m}$ 
can be spread over multiple 
processors.  See~\cite{aristoff2015parallel,aristoff2014parallel,le2012mathematical,wang2018stationary} for related results 
in special cases. 
The parallel step is more 
efficient than direct, or serial, simulation 
provided the computational 
effort to sample the QSD is 
small relative to the effort 
to simulate an escape from $U$ 
via serial simulation.

We actually do not 
need to assume that $U$ is open and
that $X(t)_{t \ge 0}$ has the {strong} 
Markov property and 
c{\`a}dl{\`a}g paths to prove 
consistency of the parallel 
step in Algorithm~\ref{alg_gen}. 
Indeed, Algorithm~\ref{alg_gen}
is consistent in the
sense of Theorems~\ref{theorem_dyn}-\ref{theorem_avg}  
whenever $X(t)_{t \ge 0}$ is a time 
homogeneous Markov process and~\eqref{assump1}-\eqref{assump2} 
hold. 
However, to combine the parallel 
step with a decorrelation step 
to obtain a coarse dynamics or 
stationary average, 
we want c{\`a}dl{\`a}g paths 
to ensure the escape time from an 
open set $U$ is a 
stopping time, and we need 
the strong Markov property 
so that we can start afresh 
at these stopping times. 

\begin{figure}
\includegraphics[width=9.5cm]{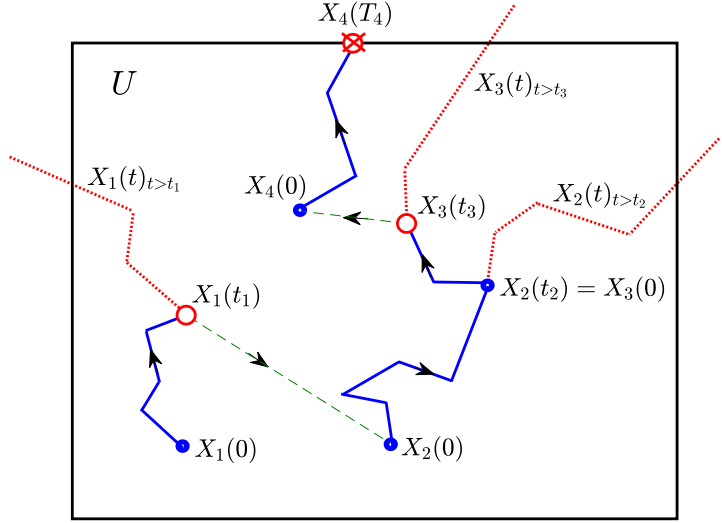}
\centering
\caption{Intuition behind the general parallel 
step of 
Algorithm~\ref{alg_gen}. Pictured 
are $L = 4$ 
trajectory fragments $(X_m(t)_{0 \le t \le t_m})_{m=1,2,3,4}$ combined 
to form one long 
trajectory,  
advancing in time in the direction 
indicated by the arrows. 
Its escape 
from $U$ is denoted with a cross.
The dashed line parts 
of the long trajectory are 
artificial and do not contribute 
to $g_{par}$ or $T_{par}$. Only 
the solid lines and solid dots 
along the long 
trajectory contribute 
to $g_{par}$ and $T_{par}$.
The dotted lines show the fragments'
irrelevant futures. Note 
that one of the fragments' 
starting point, $X_3(0)$, is equal 
to another fragment's terminal point, 
$X_2(t_2)$. There can be other 
fragments but they are not relevant 
to the parallel step in this 
example since $T_4 < t_4$.}
\label{fig_intuition}
\end{figure}

\section{Synchronous 
and asynchronous computing}\label{sec:asynch} 
Recall that the speedup in ParRep 
comes from computing the trajectory fragments 
$X_m(t)_{0 \le t \le t_m}$ partly 
or fully in parallel. These 
fragments must be {\em ordered}, 
via the index $m \ge 1$, to 
obtain the long trajectory 
pictured in Figure~\ref{fig_intuition}. 
Below, we explore two possible
ways to 
order the fragments, 
depending on whether 
we want to employ
{\em synchronous} 
or {\em asynchronous 
computing}. In the 
former case, we have 
in mind
a computing environment 
consisting of $R$ 
processors that are 
nearly synchronous. 
In the latter case 
we consider an 
arbitrary number of 
processors 
that potentially 
have widely different 
performance. 

Below, we will consider only 
fragments of constant time 
length, $t_m \equiv \Delta t$. 
For synchronous 
computing, following 
ideas
from~\cite{aristoff2015parallel,aristoff2014parallel,wang2018stationary}, we 
consider the 
ordering of trajectory 
fragments in Proposition~\ref{prop_synch} below. 

\begin{figure}
\includegraphics[width=12cm]{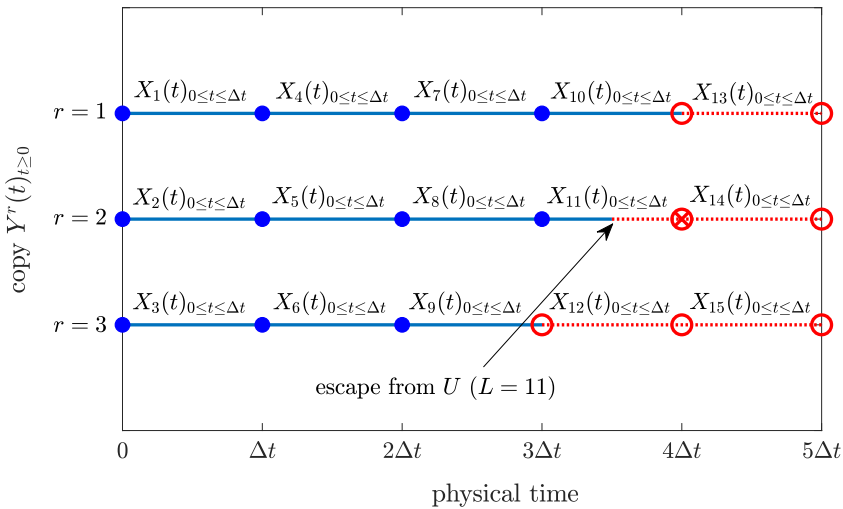}
\vskip-2pt
\caption{An example of a synchronous ParRep algorithm based on the ordering 
of fragments in Proposition~\ref{prop_synch}.
Solid dots and hollow circles correspond to fragments' initial 
and terminal points. The cross corresponds 
to the terminal 
point of the first fragment, 
in terms of the ordering, 
to escape from $U$. Times of copies that 
contribute to $g_{par}$ 
and $T_{par}$ are pictured with 
solid line segments and solid dots, while 
times that do not contribute are pictured with dotted lines and hollow circles.}
\label{fig_synchronous}
\end{figure}

\begin{proposition}[Synchronous computing]\label{prop_synch}
Suppose $Y^r(t)_{t \ge 0}$, 
$r=1,\ldots,R$, are independent 
copies of $X(t)_{t \ge 0}$ 
with ${\mathcal L}(X(0)) = \rho$. 
Let $m_k = \lfloor (k-1)/R\rfloor$ and $r_k = k-R\lfloor (k-1)/R\rfloor$, and for $k \ge 1$ define 
trajectory fragments $X_k(t)_{0 \le t \le \Delta t}$ by
\begin{equation}\label{frag1}
X_k(t) = Y^{r_k}(m_k\Delta t + t), \qquad 0 \le t \le \Delta t.
\end{equation}
Then Assumption~\ref{A1} holds with an appropriate 
definition of the trajectory
fragments' irrelevant 
futures. Thus the conclusions of Theorem~\ref{theorem_dyn} and~\ref{theorem_avg} hold.
\end{proposition}
Figure~\ref{fig_synchronous} shows 
the trajectory fragments defined in~\eqref{frag1}.

In asynchronous 
computing, perhaps the most natural ordering 
is the {\em wall-clock ordering}: 
the fragments are ordered 
according to the wall-clock 
time that their starting points 
are computed. See Figure~\ref{fig_asynchronous}.
When does the wall-clock ordering satisfy Assumption~\ref{A1}? 
Note that~\eqref{assump1} simply says that each 
fragment evolves forward in time independently 
of the preceding fragments and 
their irrelevant futures. 
This condition is easy to establish 
with
an appropriate choice of the 
irrelevant futures.
Ensuring~\eqref{assump2} holds is more subtle. We will 
show, however, that 
if the wall-clock time it takes 
to compute each fragment depends on processor 
variables, but not on the fragments themselves, 
then the wall-clock time ordering is independent of 
the fragments and~\eqref{assump2} holds.

We will distinguish 
between a wall-clock time and 
a {physical time}, where the former 
is self-explanatory and the
latter refers to the time index 
$t$ of a copy of $X(t)_{t \ge 0}$. 
Let $Y^r(t)_{t \ge 0}$ be independent 
copies of 
$X(t)_{t \ge 0}$ starting at 
independent samples of the 
QSD $\rho$. The wall-clock time 
ordering of fragments satisfies 
Assumption~\ref{A1} above if {\em (i)} the wall-clock times are independent 
of the physical times, 
{\em (ii)} the wall-clock time 
to compute copy $Y^r(t)_{t \ge 0}$ 
is an increasing function of 
the physical time, and 
{\em (iii)} two processors 
never finish at exactly the same 
wall-clock time, so that the wall-clock 
times can be given a 
unique ordering. Write $t_{wall}^r(m)$ 
for the wall-clock time it takes to 
compute $Y^r(t)_{t \ge 0}$ up to 
physical time $t = m\Delta t$. Proposition~\ref{prop_asynch} below makes 
the claims above precise:

\begin{figure}
\includegraphics[width=12cm]{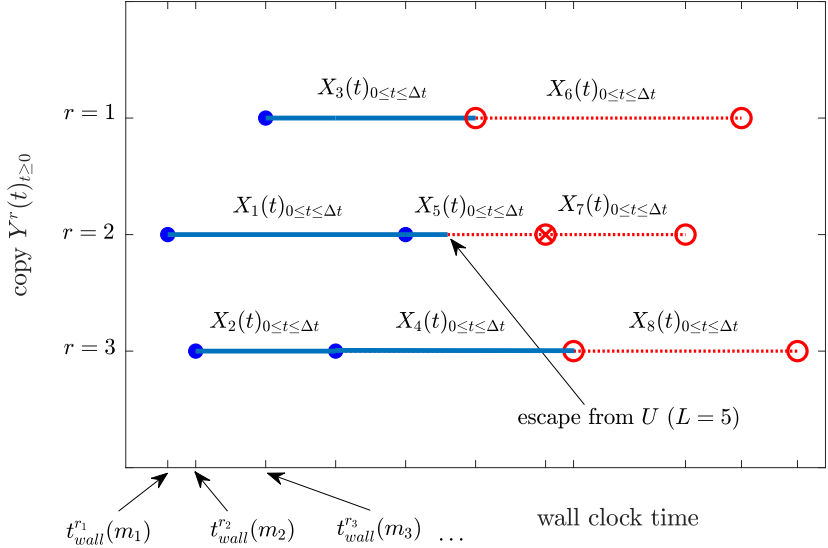}
\vskip-2pt
\caption{An example of an asynchronous ParRep algorithm based on wall-clock time ordering of fragments in 
Proposition~\ref{prop_asynch}. Solid dots and hollow circles correspond to fragments' initial 
and terminal points. The cross corresponds 
to the terminal 
point of the first fragment, 
in terms of the ordering, 
to escape from $U$. Times of copies that 
contribute to $g_{par}$ 
and $T_{par}$ are pictured with 
solid line segments and solid dots, while 
times that do not contribute are pictured with dotted lines and hollow circles.}
\label{fig_asynchronous}
\end{figure}

\begin{proposition}[Asynchronous computing]\label{prop_asynch} 
Suppose $Y^r(t)_{t \ge 0}$, 
$r=1,\ldots,R$, are independent 
copies of $X(t)_{t \ge 0}$ 
with ${\mathcal L}(X(0)) = \rho$. 
Assume $(t_{wall}^r(m)_{m \ge 0})^{1 \le r \le R}$ are nonnegative random numbers such that: 
\vskip5pt
\begin{itemize}
\item[(i)]
$(t_{wall}^r(m)_{m \ge 0})^{1 \le r \le R}$ is independent of $(Y^r(t)_{t \ge 0})^{1\le r \le R}$; 
\vskip2pt
\item[(ii)] Almost surely, $t_{wall}^r(m) \le t_{wall}^r(n)$ when 
$m \le n$ and $1 \le r \le R$;
 \vskip2pt
\item[(iii)] Almost surely, 
there is a unique sequence
 $(r_k,m_k)_{k \ge 1}$ such that:
 \begin{align*}
     &(r_k,m_k)_{k \ge 1} \text{ has range } \{1,\ldots,R\}\times\{0,1,2,\ldots\} 
     \qquad \text{(surjectivity)},\\
     &t_{wall}^{r_1}(m_1)<t_{wall}^{r_2}(m_2)<t_{wall}^{r_3}(m_3)<\ldots \qquad\qquad 
     \text{(monotonicity)}.
 \end{align*}
\end{itemize}
\vskip5pt
For $k \ge 1$ define trajectory fragments $X_k(t)_{0 \le t \le \Delta t}$ by
\begin{equation}\label{frag2}
X_k(t) = Y^{r_k}(m_k\Delta t + t), \qquad 0 \le t \le \Delta t.
\end{equation}
Then Assumption~\ref{A1} holds 
with an appropriate definition 
of the trajectory fragments' irrelevant futures.
Thus the conclusions of 
Theorems~\ref{theorem_dyn} and~\ref{theorem_avg} hold.
\end{proposition}

Figure~\ref{fig_asynchronous} 
shows the trajectory fragments $X_m(t)_{0 \le t \le t_m}$ defined in~\eqref{frag2}. 

Assumptions {\em (ii)} and {\em (iii)} are 
quite natural, but assumption
{\em (i)} can 
fail in many very ordinary settings. We 
sketch an example explaining 
how this could happen. 
Suppose
$U = (0,1) \subseteq {\mathbb R}$ 
and say $X(t)_{t \ge 0}$ obeys some one dimensional stochastic differential equation. 
Suppose we use an integrator 
for $X(t)_{t \ge 0}$ that is 
slow near $0$ but fast near 
$1$. Then the wall-clock ordering 
will likely put trajectory fragments 
that are near $1$ 
ahead of those near $0$. This 
bias in the ordering 
would in turn create a bias 
toward escaping through $1$: 
in Algorithm~\ref{alg_gen}, we 
would expect that ${\mathbb P}(X_{par}= 1) > {\mathbb P}(X_T = 1)$, where $X_T$ is the correct escape point. We 
construct a specific 
example demonstrating 
this bias in Remark~\ref{rmk_bias} 
in Section~\ref{sec:proofs2} below.

The speed of integrators does 
commonly depend on position in 
space, particularly when 
the time step varies to account 
for numerical stiffness~\cite{talay1994numerical}. This is an important caveat 
to keep in mind for asynchronous 
algorithms. This  
issue has not been explored 
much in the literature; see 
however brief discussions in~\cite{le2012mathematical} and~\cite{perez2017parsplice}.

The setting of Parsplice~\cite{perez2017parsplice,perez2018long,swinburne2018self} is slightly 
different from the above. In ParSplice, 
a {\em splicer} tells a {\em producer} 
to generate fragments among several 
metastable sets. The splicer 
distributes the fragments 
according to where it speculates 
that they will be needed.
These 
fragments are given a 
label
as soon as they are assigned, 
and this label never changes. The 
labels are assigned 
in wall-clock time 
order. Thus label 
$i$ is less than label $j$ if and only if 
the splicer tells 
the producer to generate
fragment $i$ before it 
tells the producer
to generate fragment $j$.
When the splicer tells the producer 
to generate a fragment 
in a particular metastable 
set $U$, it takes as 
its starting point the 
terminal point of the fragment 
in $U$ with the smallest label. 
Crucially, this label 
is smallest among {\em all} fragments 
in $U$ and not just among 
fragments in $U$ which have been 
fully computed at the current 
wall-clock time. 
Thus, the ordering of fragments 
in $U$, fixed by the splicer, can 
be seen as independent of 
the fragments themselves, 
and the arguments above demonstrate 
consistency of ParSplice.

\section{PDMPs}\label{sec:PDMPs}

The remainder of this article will focus on 
applying our ideas above to PDMPs. 
We begin with a brief informal
description of PDMPs. 
A PDMP is a 
c\`adl\`ag process consisting 
of a deterministic dynamics 
interrupted by jumps at 
random times; formally, a PDMP in ${\mathbb R}^d$ has a generator $L$ of the form
\begin{equation}\label{PDMP_gen}
Lf(z) = \partial_{\Gamma(z)} f(z) + \lambda(z) \int (f(z')-f(z))Q(z,dz')
\end{equation}
acting on suitable $f:{\mathbb R}^d \to {\mathbb R}$. Here $\lambda(z)$ is
the jump rate at $z \in {\mathbb R}^d$, a Markov kernel
$Q(z,dz')$ describes the jump distribution, and $\Gamma$ defines 
the deterministic flow $$\partial_t \psi(t,z) = \Gamma(\psi(t,z)),\qquad \psi(0,z) = z.$$

Write $\theta_0 + \ldots + \theta_{n-1}$ for the $n$th jump time, so that
$\theta_{n-1}$ is the 
holding time before the $n$th  
jump, and 
write $\xi_n$ for the position immediately after the 
$n$th jump, with $\xi_0$ the initial position. Then 
the PDMP generated by~\eqref{PDMP_gen} is described by $\psi$ together 
with 
$(\xi_n,\theta_n)_{n \ge 0}$; 
we call the latter the {\em skeleton chain} of 
$Z(t)_{t \ge 0}$. Note that the 
skeleton chain is a
time homogeneous Markov chain.

For convenience we describe a 
way to simulate a PDMP described 
by~\eqref{PDMP_gen} in
Algorithm~\ref{alg1} below. In 
the algorithm, we 
abuse notation by writing 
$\theta_n$, $\xi_n$, and 
$Z(t)$ for particular 
realizations of these 
random objects.

\begin{algorithm}
\caption{Simulating a PDMP}
Starting from an initial point $\xi_0$ and time $t= 0$, set $n = 0$, and iterate:

{\bf 1.} Sample $\theta_n$ according to the distribution 
\begin{equation}\label{time_dist}
{\mathbb P}(\theta_n > r) = \exp\left(-\int_0^r \lambda(\psi(s,\xi_{n}))\,ds\right).
\end{equation}

{\bf 2.} Set $Z(t+s) = \psi(s,\xi_{n})$ for $s \in [0,\theta_n)$, and sample $\xi_{n+1}$ 
from $Q(Z(t+\theta_n^-),dz)$.

{\bf 3.} Update $t \leftarrow t+\theta_n$ and then $n \leftarrow n+1$. Then return to Step 1.

Steps 1-3 above define a realization of $Z(t)_{t \ge 0}$ with skeleton chain $(\xi_n,\theta_n)_{n\ge 0}$.\label{alg1}
\end{algorithm}

Sampling the times $\theta_n$ is a nontrivial task, but there 
are efficient methods 
based on Poisson 
thinning~\cite{bierkens2016zig,vanetti2017piecewise} and 
identifying certain critical points along the flow direction~\cite{harland2017event}. See also~\cite{vanetti2017piecewise} for other 
methods to simulate a PDMP, including some 
based on time discretization. 
We will 
always assume our initial 
points $(\xi_0,\theta_0)$ 
are chosen so that $\theta_0$ 
satisfies~\eqref{time_dist} 
for $n = 0$, so that we can skip the first step in 
Algorithm~\ref{alg1}.


\subsection{Example: linear flow}

Consider a PDMP with deterministic paths that are lines in ${\mathbb R}^{d-1}$ corresponding 
to a finite collection of velocity vectors $d_i \in {\mathbb R}^{d-1}$, $i \in {\mathcal I} \subseteq {\mathbb N}$. Its generator $L$ is defined on suitable functions $f:{\mathbb R}^{d-1} \times {\mathcal I} \to {\mathbb R}$ by
\begin{equation}\label{gen}
Lf(x,i) = d_i \cdot  \nabla f(x,i) + \sum_{j\ne i} \lambda_j(x,i)(f(x,j)-f(x,i)),
\end{equation}
where $\lambda_j(x,i)\ge 0$ for $j \ne i$. Suppose we want to sample the probability density
\begin{equation}\label{Boltzmann}
Z^{-1}e^{-V(x)},\qquad Z = \int e^{- V(x)}\,dx,
\end{equation}
where $V:{\mathbb R}^{d-1} \to {\mathbb R}$ is 
smooth and grows 
sufficiently fast at $\infty$ 
so that $Z<\infty$. For the PDMP generated by~\eqref{gen} to have an invariant probability density independent of $i$ and proportional to~\eqref{Boltzmann}, 
the jump rates must satisfy
\begin{equation}\label{ratebalance}
\sum_{j\ne i} \left(\lambda_i(x,j) - \lambda_j(x,i)\right) = -  d_i \cdot \nabla  V(x).
\end{equation}
See Remark~\ref{prop_balance} in Section~\ref{sec:proofs2} below, 
and~\cite{bierkens2016zig} for a 
similar calculation.

Event 
Chain Monte Carlo and the Zig-Zag process fit into this 
framework. And while these methods 
were designed for efficient sampling, 
we argue that they may be 
limited in certain situations. 
To see why, note that~\eqref{ratebalance} says that at $x$, the rate into $d_i$ minus the rate out of $d_i$ equals minus 
the gradient of $V$ in direction $d_i$. Thus the PDMP is likely to change 
directions when it moves up a steep slope of $V$. This suggests the PDMP 
can struggle to escape from a 
basin of attraction of $V$, 
defined as the set 
of initial conditions $x(0)$ for which $dx(t)/dt = -\nabla V(x(t))$ has a unique long-time limit.

\section{ParRep for PDMPs}\label{sec:parrep_pdmp}

In this section $Z(t)_{t \ge 0}$ is 
a PDMP with stationary 
distribution $\pi$, and 
$f$ is a real-valued function 
defined on the state space 
of $Z(t)_{t \ge 0}$.
Below we outline some 
ParRep algorithms 
for estimating coarse dynamics 
as well as stationary averages, 
with a focus on the latter. The 
stationary average of $f$ is 
\begin{equation}\label{ergodic}
\langle f \rangle = \int f(x)\pi(dx).
\end{equation}

Algorithms~\ref{alg_main}  and~\ref{alg_main_dyn} 
below are ParRep algorithms based on the skeleton chain 
and the continuous time PDMP, respectively. 
Algorithms~\ref{alg_parstep1} and~\ref{alg_parstep3} 
are parallel steps for synchronous 
computing,
while 
Algorithms~\ref{alg_parstep2} and~\ref{alg_parstep4} 
are for asynchronous computing. 
Algorithms~\ref{alg_parstep1} and~\ref{alg_parstep3}, 
which are essentially 
extensions to PDMPs of 
algorithms recently proposed 
for continuous time Markov chains~\cite{wang2017parallel,wang2018stationary}, use the ordering 
of trajectory fragments defined in 
Proposition~\ref{prop_synch}. 
Algorithms~\ref{alg_parstep2} and~\ref{alg_parstep4} employ the 
wall-clock time ordering of 
fragments from Proposition~\ref{prop_asynch}. 
We prove consistency 
of all of our parallel steps 
via our trajectory 
fragment framework.

We do not 
attempt to prove existence, uniqueness 
or convergence to the QSD for general PDMPs. Instead we 
refer the reader 
to recent articles~\cite{champagnat2016exponential,champagnat2017general} 
for conditions which 
ensure convergence 
to a unique QSD. 
From those works,
under appropriate assumptions,
one can establish  convergence 
to a unique 
QSD in $D \times {\mathcal I} \subseteq {\mathbb R}^{d}$
for a PDMP generated by~\eqref{gen}.
For 
instance, exponential 
convergence 
is guaranteed if $D\subseteq {\mathbb R}^{d-1}$ is an open
connected bounded domain 
and there exist 
$m,M$ so that $0<m\le \lambda_j(x,i)\le M$ for 
all $x \in D$ and $i,j \in {\mathcal I}$: see~\cite{champagnat2016exponential}, pg. 261.
Similar arguments can be made for the 
QSD of the skeleton chain.
Even without theoretical 
guarantees, in practice, 
one can empirically 
validate convergence 
to the QSD using 
certain diagnostics; 
see for instance~\cite{binder2015generalized}.

\subsection{Skeleton chain-based parrep algorithm}\label{sec:stat_skel}

Let ${\mathcal W}$ be the 
collection of metastable sets for 
the skeleton chain 
$(\xi_n,\theta_n)_{n \ge 0}$. For 
instance, if
$Z(t)_{t \ge 0}$ has generator 
similar to the form~\eqref{gen} 
and we want to sample from 
the distribution~\eqref{Boltzmann},
it is natural to define ${\mathcal W}$ 
in terms of basins of attraction of $V$, 
in which case elements of ${\mathcal W}$ may be identified on the 
fly by gradient descent~\cite{voter1998parallel,voter2012accelerated}. See Section~\ref{sec:numerics} for an example of metastable sets defined this way.
\begin{assumption}\label{QSD_skel}
$(\xi_n,\theta_n)_{n \ge 0}$
has a QSD $\nu = \nu_W$ in each $W\in {\mathcal W}$ satisfying 
\begin{equation*}
\nu(A) = \lim_{n \to \infty} {\mathbb P}((\xi_n,\theta_n) \in A| (\xi_m,\theta_m) \in W,\,0\le m \le n) \quad \forall \,A \subseteq W.
\end{equation*}
\end{assumption}

For simpler notation, we
do not explicitly indicate the dependence of $\nu$ on $W$.

\begin{algorithm}
\caption{Synchronous skeleton chain parallel step in $W$} 
{\bf 1.} Generate iid samples
$(\xi_0^r,\theta_0^r)^{r=1,\ldots,R}$ from the 
QSD $\nu$ in $W$. Using these as starting points, 
independently evolve $R$ copies $((\xi_n^r,\theta_n^r)_{n \ge 0})^{r=1,\ldots,R}$ of the skeleton chain. 

{\bf 2.} Let $N = \inf\{n : \exists\,r\,s.t.\,(\xi_n^r,\theta_n^r) \notin W\}$, 
$J = \min\{r :(\xi_N^r,\theta_N^r) \notin W\}$, and define
\begin{equation*}
f_{par} =  \sum_{n=0}^{N-2}\sum_{r=1}^R \int_0^{\theta_n^r}f(\psi(t,\xi_{n}^r))\,dt + \sum_{r=1}^J \int_0^{\theta_{N-1}^r}f(\psi(t,\xi_{N-1}^r))\,dt
\end{equation*}
and $T_{par} = {\mathbbm{1}}_{par}$ by using the same formula but with ${\mathbbm{1}}(z) \equiv 1$ in place of $f$. Set $$(\xi_{par},\theta_{par}) = (\xi_N^J,\theta_N^J).$$
Once $f_{par}$, $T_{par}$ and $(\xi_{par},\theta_{par})$ can be computed, the parallel step is complete.\label{alg_parstep1} 
\end{algorithm}

\begin{algorithm}
\caption{Asynchronous skeleton chain parallel step in $W$}

{\bf 1.} Generate iid samples
$(\xi_0^r,\theta_0^r)^{r=1,\ldots,R}$ from the 
QSD $\nu$ in $W$. Using these as starting points, 
independently evolve $R$ copies $((\xi_n^r,\theta_n^r)_{n \ge 0})^{r=1,\ldots,R}$ of the skeleton chain.

{\bf 2.} Reorder these 
skeleton chain points in 
the 
order they are computed in wall-clock time, {\em i.e.} as $(\xi_{m_k}^{r_k},\theta_{m_k}^{r_k})_{k\ge 1}$ where $t_{wall}^{r_1}(m_1)\le t_{wall}^{r_2}(m_2) \le \ldots$ 
and $t_{wall}^r(n)$ is
the wall-clock time it takes to compute the skeleton chain $(\xi_m^r,\theta_m^r)_{m \ge 0}$ up to physical time $m = n$.
Set $\sigma^r = \inf\{m :(\xi_m^r,\theta_m^r) \notin W\}$, $K = \inf\{k:\sigma^{r_k} \le m_k+1\}$, and 
\begin{equation*}
f_{par} = \sum_{k=1}^{K} \int_0^{\theta_{m_k}^{r_k}}f(\psi(t,\xi_{m_{k}}^{r_k}))\,dt.
\end{equation*}
Define $T_{par} = {\mathbbm{1}}_{par}$ by using the same formula but with ${\mathbbm{1}}(z) \equiv 1$ in place of $f$. Let $$(\xi_{par},\theta_{par}) = (\xi_{\sigma^{r_K}}^{r_K},\theta_{\sigma^{r_K}}^{r_K}).$$
Once $f_{par}$, $T_{par}$ and $(\xi_{par},\theta_{par})$ can be computed, the parallel step is complete.\label{alg_parstep2}
\end{algorithm}

Algorithms~\ref{alg_parstep1} and~\ref{alg_parstep2} are {parallel steps} 
designed for synchronous 
and asynchronous computing, respectively. See Figure~\ref{fig_parstep1} 
for a diagram of both parallel steps. 
The first step in Algorithm~\ref{alg_parstep1} and 
Algorithm~\ref{alg_parstep2} -- 
called {\em dephasing} in the literature~\cite{aristoff2015parallel,aristoff2014parallel,voter1998parallel,wang2018stationary} -- involves generating 
$R$ independent samples from the QSD $\nu$ in $W$. 
These QSD samples may be 
obtained in a variety of ways. One option is to do rejection sampling 
using independent copies of the skeleton chain: whenever a copy 
escapes from $W$, start it afresh in $W$ until each
copy has remained in $W$ for a long enough consecutive time. Another possibility 
is based on the Fleming-Viot branching process~\cite{binder2015generalized,ferrari2007quasi}: when a copy escapes from $W$, 
restart it at the current position of a 
copy still in $W$ chosen at random. For more discussion see~\cite{binder2015generalized,simpson2013numerical,voter1998parallel}.

\begin{figure}
\subfigure{\includegraphics[width=6.45cm]{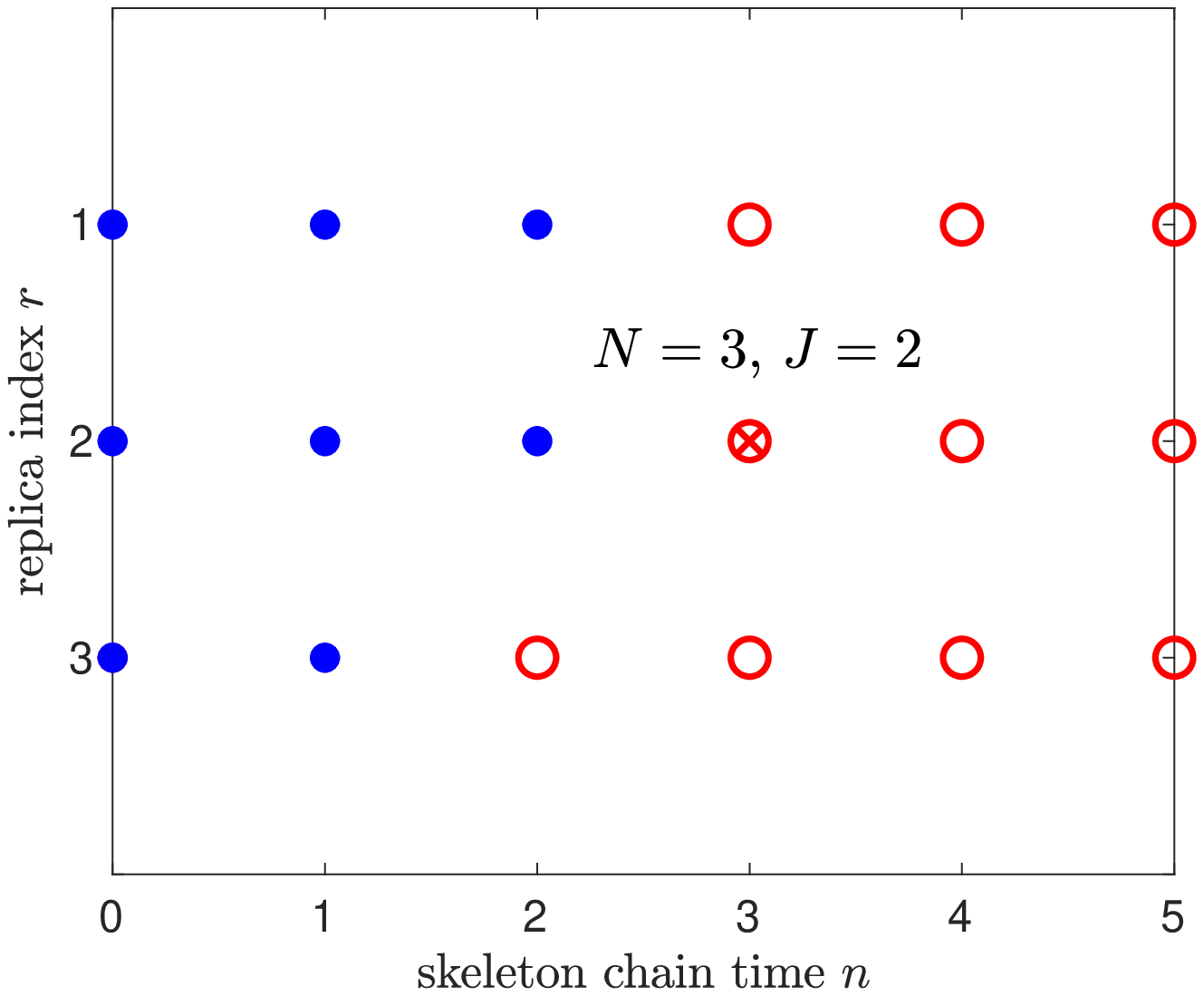}}
\subfigure{\includegraphics[width=6.45cm]{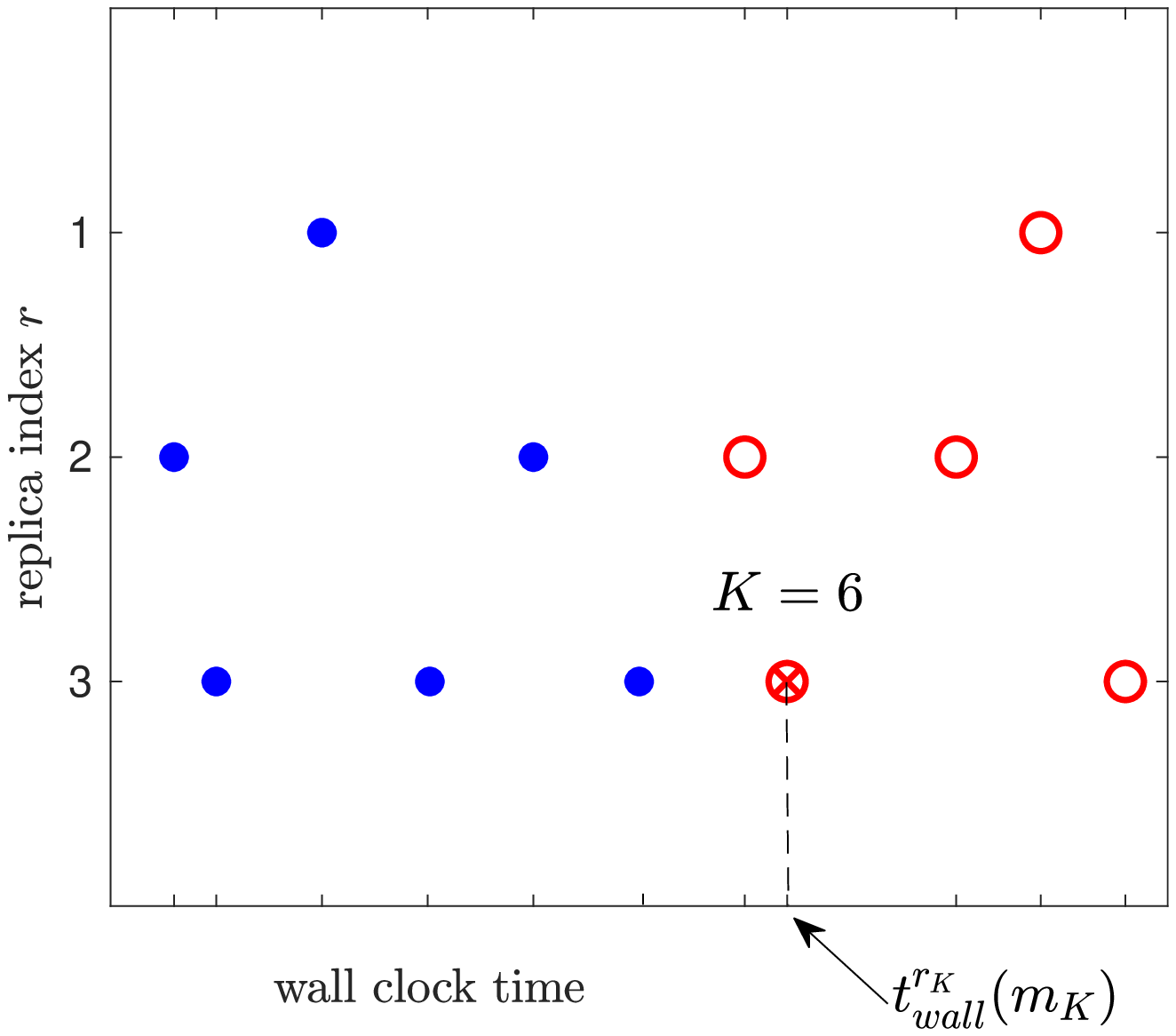}}

\caption{Illustration of the parallel steps used in Algorithm~\ref{alg_main}. The number of copies, or parallel replicas, is $R = 3$. 
The cross indicates an escape from $W$. 
Time steps of copies 
of the skeleton chain that contribute 
to $f_{par}$ and $T_{par}$ are pictured 
with solid dots, while time steps that 
do not contribute are pictured with 
hollow circles. 
Left: The 
synchronous parallel 
step, Algorithm~\ref{alg_parstep1}. 
Copy 
$r = 2$ escapes from $W$ 
at skeleton chain time $3$. 
In this example $N = 3$ and $J = 2$.
Right: The 
asynchronous parallel 
step, Algorithm~\ref{alg_parstep2}. 
Copy $r=3$ 
is the first to escape from $W$ in terms 
of the wall-clock time ordering. It escapes at 
wall-clock time $t_{wall}^{r_K}(m_K)$. In this 
example $K = 6$.}
\label{fig_parstep1}
\end{figure}

Recall the speedup from ParRep comes from 
the parallel step. The speedup -- 
the factor by which ParRep reduces the wall-clock computation time, compared to 
serial simulation of a trajectory 
of the same physical time -- can be a factor of up to $R$, the number of copies or replicas~\cite{aristoff2015parallel,aristoff2014parallel, le2012mathematical,voter1998parallel,wang2018stationary}, 
when Algorithm~\ref{alg1} is used to simulate the skeleton chain. See Figure~\ref{fig1}.
The parallel step is consistent no matter the 
choice of $W$, but if $W$ is not 
metastable there 
may be no gain in efficiency, as too much computation time 
will be spent sampling the QSD.

\begin{algorithm}
\caption{Skeleton chain computation of stationary averages}
Choose an initial point $(\xi_0,\theta_0)$, set $f_{sim} = 0$, 
$T_{sim} = 0$, and iterate:
\vskip5pt

{\bf 1.} Starting at $(\xi_0,\theta_0)$, evolve $(\xi_n,\theta_n)_{n\ge 0}$ forward in time, stopping at time 
\begin{equation*}
L = \inf\{n \ge T_{corr}^\nu(W)-1:\exists\,W\in {\mathcal W}\,s.t.\,(\xi_{n-k},\theta_{n-k}) \in W,\,k=0,\ldots,T_{corr}^\nu(W)-1\},
\end{equation*}
the first time it remains in some $W \in {\mathcal W}$ for 
$T_{corr}^\nu(W)$ 
consecutive time steps. Set 
\begin{equation*}
f_{decorr} = \sum_{n=0}^{L-1} \int_0^{\theta_{n}}f(\psi(t,\xi_{n}))\,dt 
\end{equation*}
and $T_{decorr} = {\mathbbm{1}}_{decorr}$ using the same formula. Store this $W$ for Step 2 and update $$f_{sim} \leftarrow f_{sim} + f_{decorr}, \qquad T_{sim} \leftarrow T_{sim} + T_{decorr}.$$

{\bf 2.} Run the parallel step (Algorithm~\ref{alg_parstep1} or~\ref{alg_parstep2}) in the set $W$ from Step 1. Update
$f_{sim} \leftarrow f_{sim} + f_{par}$, $T_{sim} \leftarrow T_{sim} + T_{par}$, 
$(\xi_0,\theta_0) = (\xi_{par},\theta_{par})$, and return to Step 1.
\vskip5pt

The algorithm stops when $T_{sim}$ exceeds a user-chosen threshold $T_{stop}$. At this time, 
\begin{equation*}
\langle f \rangle \approx \frac{f_{sim}}{T_{sim}}
\end{equation*}
is our estimate of the stationary average~\eqref{ergodic}. \label{alg_main}
\end{algorithm}

Theorem~\ref{thm_consistent} 
gives conditions that establish 
consistency of  Algorithm~\ref{alg_parstep1} and~\ref{alg_parstep2}. 
For Algorithm~\ref{alg_parstep2},
the crucial condition is that the wall-clock times it takes 
for the processors to compute steps of the 
skeleton chains are independent of those 
chains. Whether this holds true will depend on the algorithm used to simulate the PDMP. 
If it is a time discretization-based algorithm, or an implementation 
of Algorithm~\ref{alg1} based 
on Poisson thinning, then the computational effort to 
obtain one step of the skeleton 
chain can be larger in regions 
in state space with lower jump rates.
For CTMCs simulated via the SSA/Gillespie  
algorithm~\cite{anderson2011continuous}, the effort 
to simulate one step of the 
skeleton chain may be
essentially independent 
of the position of the chain. 

\begin{theorem}[Consistency of the parallel steps  Algorithm~\ref{alg_parstep1} and~\ref{alg_parstep2}]\label{thm_consistent}
\hskip150pt
\vskip2pt
\noindent (i) Let $\nu$ be the QSD of $(\xi_n,\theta_n)_{n \ge 0}$ in some $W\in {\mathcal W}$, suppose that ${\mathcal L}(\xi_0,\theta_0) = \nu$, and define $M = \inf\{n\ge 0:(\xi_n,\theta_n) \notin W\}$. Then in Algorithm~\ref{alg_parstep1},
\begin{equation}\label{avg_correct2}
{\mathbb E}(f_{par}) = {\mathbb E}\left(\sum_{n=0}^{M-1} \int_0^{\theta_n}f(\psi(t,\xi_{n}))\,dt \right)
\end{equation} 
and
\begin{equation}\label{exit_correct2}
 (\xi_{par},\theta_{par}) \stackrel{\mathcal L}{=} (\xi_M,\theta_M).
\end{equation}
\vskip5pt
\noindent (ii) 
Suppose  
$(t_{wall}(m)_{m \ge 0})^{1 \le r \le R}$, the wall-clock times from Algorithm~\ref{alg_parstep2}, satisfy
the assumptions 
of Proposition~\ref{prop_asynch} when $(Y^r(t)_{t \ge 0})^{r=1,\ldots,R}$ 
equals $((\xi_n^r,\theta_n^r)_{n \ge 0})^{r=1,\ldots,R}$. Adopt the 
assumptions in (i) above. Then~\eqref{avg_correct2}-\eqref{exit_correct2} hold.
\end{theorem}

The times $T_{corr}^\nu(W)$ in Algorithm~\ref{alg_main} may be 
 chosen on the fly, or they may be 
 set at the beginning of simulations. Choosing an appropriate value may be done using various
 convergence diagnostics or 
 a priori information; see~\cite{binder2015generalized,simpson2013numerical,voter1998parallel} for details.

Consistency of 
the parallel steps, 
together with exactness 
of the decorrelation step,
show that Algorithm~\ref{alg_main} 
produces correct stationary 
averages, provided 
some mild recurrence 
assumptions hold~\cite{aristoff2015parallel}. The 
reason is 
essentially the law of 
large numbers: for 
computations of stationary 
averages, due 
to repeated visits 
to each metastable set, 
in the parallel steps
it is enough to get contributions
$f_{par}$ 
to $f_{sim}$ with
the correct {\em average} value 
along with
escape events with 
the correct law.

We do not attempt here to prove 
ergodicity using this argument, 
but mention it
has been studied previously 
in~\cite{aristoff2015parallel,wang2018stationary}. 
Our numerical simulations in 
Section~\ref{sec:numerics} below 
also support its validity. One interesting aspect of the 
parallel step is that  
the averaging over independent copies or replicas  
can be considered a bonus, as it likely lowers 
the variance of the estimate 
$f_{sim}/T_{sim} \approx \langle f \rangle$ of the stationary average, compared to an estimate
from a serial trajectory of physical 
time length~$T_{sim}$.

\subsection{Continuous time PDMP-based algorithm}\label{sec:stat_PDMP}

Let ${\mathcal V}$ be the 
collection of metastable sets for 
$Z(t)_{t \ge 0}$. As above, if
$Z(t)_{t \ge 0}$ has a generator 
similar to~\eqref{gen} 
and we want to sample 
from the distribution~\eqref{Boltzmann}, 
the elements of ${\mathcal V}$ 
can be defined in terms of the basins of attraction of $V$. 
We will require a 
time interval 
$\Delta t>0$, which 
is not necessarily 
a time step for discretizing the PDMP. 
For instance, $\Delta t$ could 
be a polling time for resynchronizing parallel processors. 

\begin{algorithm}
\caption{Synchronous continuous time parallel step in $W$}  
{\bf 1.} Generate iid samples
$Z^r(0)^{r=1,\ldots,R}$ from the 
QSD $\mu$ in $W$. Using these as starting points, 
independently evolve $R$ copies $(Z^r(t)_{t \ge 0})^{r=1,\ldots,R}$ of the PDMP. 

{\bf 2.} Let $\tau^r = \inf\{t:Z^r(t)\notin W\}$, set 
$$N = \inf\{n \in {\mathbb N}:\exists\,r\,s.t.\,\tau^r \le n \Delta t\}, \qquad J = \min\{r:\tau^r \le N \Delta t\},$$ and define
\begin{align*}
f_{par} &=  \sum_{n=1}^{N-1}\sum_{r=1}^R \int_{(n-1)\Delta t}^{n\Delta t}f(Z^r(t))\,dt \\
&\qquad \qquad  + \sum_{r=1}^{J-1} \int_{(N-1)\Delta t}^{N\Delta t}f(Z^r(t))\,dt 
+ \int_{(N-1)\Delta t}^{\tau^J} f(Z^J(t))\,dt,
\end{align*}
and $T_{par} = {\mathbbm{1}}_{par}$ by using the same formula but with ${\mathbbm{1}}(z) \equiv 1$ in place of $f$. Set $$Z_{par} = Z^J(\tau^J).$$
Once $f_{par}$, $T_{par}$ and $Z_{par}$ can be computed, the parallel step is complete.\label{alg_parstep3}
\end{algorithm}

\begin{algorithm} 
\caption{Asynchronous continuous time parallel step in $W$}
{\bf 1.} Generate iid samples
$Z^r(0)^{r=1,\ldots,R}$ from the 
QSD $\mu$ in $W$. Using these as starting points, 
independently evolve $R$ copies $(Z^r(t)_{t \ge 0})^{r=1,\ldots,R}$ of the PDMP. 

{\bf 2.} Reorder the $\Delta t$ time intervals of these copies in the order they are computed in wall-clock time, {\em i.e.} as $Z^{r_k}(m_k\Delta t)_{k\ge 1}$ where $t_{wall}^{r_1}(m_1)\le t_{wall}^{r_2}(m_2) \le \ldots$ and $t_{wall}^r(n)$ is the wall-clock time 
it takes to compute 
the PDMP $Z^r(t)_{t \ge 0}$ up to physical time 
$t=n\Delta t$.  Set $\tau^r = \inf\{t: Z^r(t) \notin W\}$,  $K = \inf\{k:\tau^{r_k} \le (m_k+1)\Delta t\}$, and 
\begin{equation*}
f_{par} = \sum_{k=1}^{K-1} \int_{m_k \Delta t}^{(m_k+1)\Delta t}f(Z^{r_k}(t))\,dt + \int_{m_K\Delta t}^{\tau^{r_K}} f(Z^{r_K}(t))\,dt
\end{equation*}
and $T_{par} = {\mathbbm{1}}_{par}$ by using the same formula but with ${\mathbbm{1}}(z) \equiv 1$ in place of $f$. Let $$Z_{par} = Z^{r_K}(\tau^{r_K}).$$
Once $f_{par}$, $T_{par}$ and $\xi_{par}$ can be computed, the parallel step is complete.\label{alg_parstep4}
\end{algorithm}

We will adopt the following assumption.
\begin{assumption}\label{A0b}
$Z(t)_{t \ge 0}$
has a QSD $\mu = \mu_W$ in each $W\in {\mathcal V}$ 
satisfying 
\begin{equation*}
\mu(A) = \lim_{t \to \infty} {\mathbb P}(Z(t) \in A|Z(s) \in W,\,0 \le s \le t)\quad \forall\,A \subseteq W.
\end{equation*}
\end{assumption} 

We do not explicitly indicate the dependence of $\mu$ on $W$. 
Notice the QSD of the PDMP is different from that of its skeleton chain in general.

Algorithms~\ref{alg_parstep3} and~\ref{alg_parstep4} are {parallel steps} 
designed for synchronous 
and asynchronous computing, respectively; see Figure~\ref{fig_parstep2}.
The first step in Algorithm~\ref{alg_parstep3} and 
Algorithm~\ref{alg_parstep4} -- the dephasing step -- involves generating 
$R$ independent samples from the QSD $\mu$ in $W$. 
Note that this is the QSD of the PDMP in $W$, 
not the QSD of its skeleton chain. The 
QSD samples can be obtained exactly as 
described in the previous section, but with the PDMP taking 
the place of the skeleton chain.

\begin{figure}
\subfigure{\includegraphics[width=6.45cm]{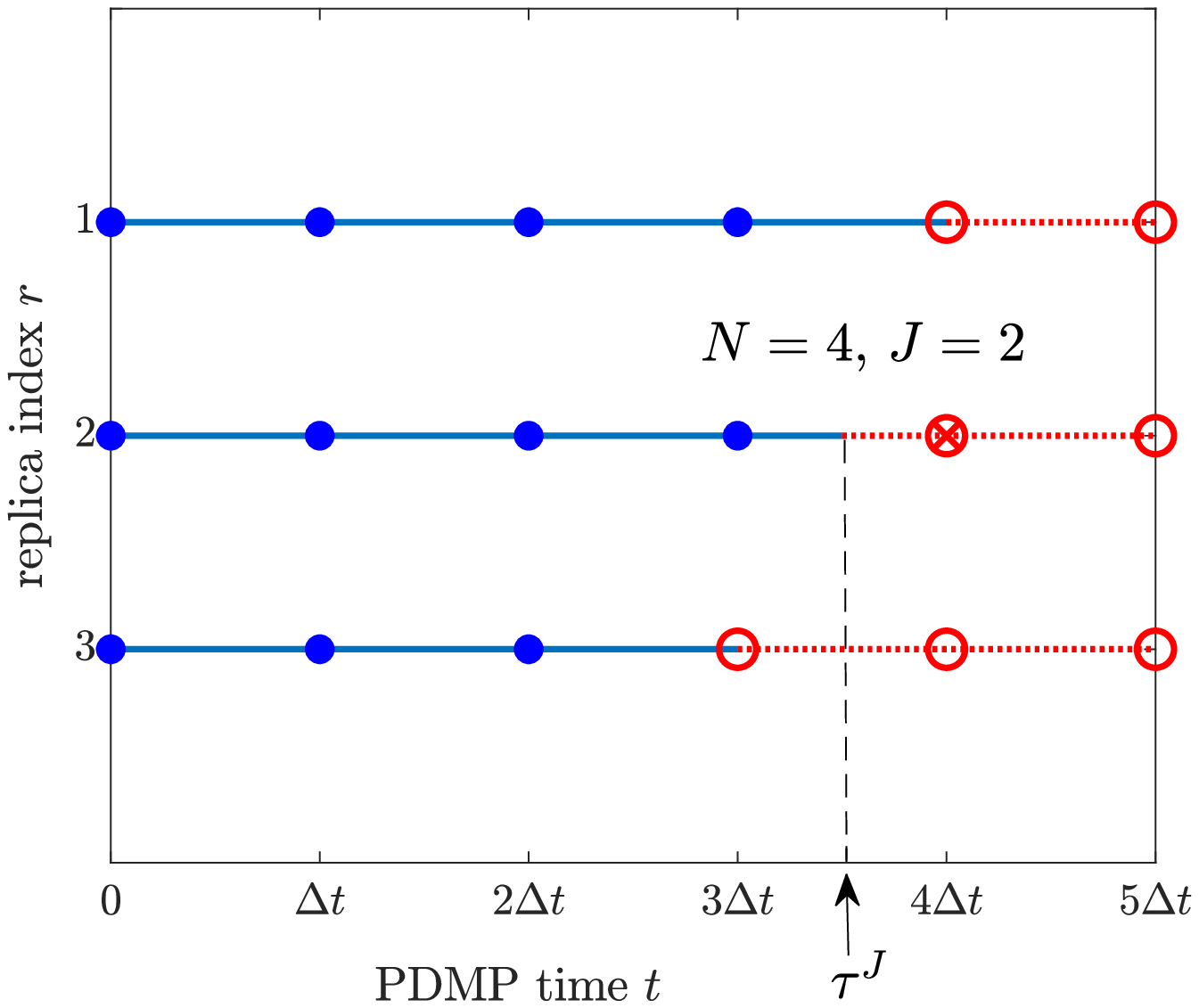}}
\subfigure{\includegraphics[width=6.45cm]{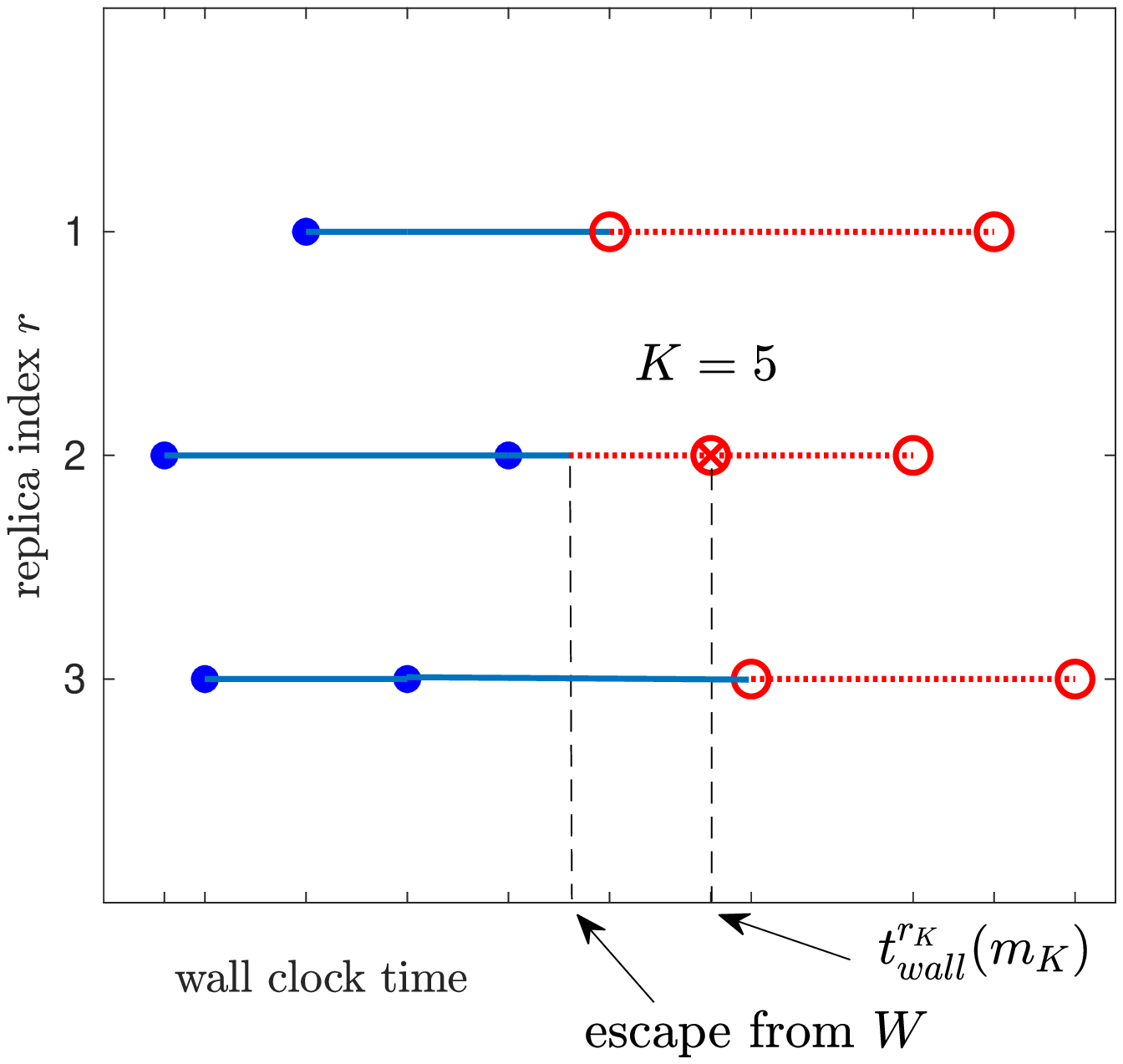}}

\caption{Illustration of the parallel steps used in Algorithm~\ref{alg_main_dyn}. The number of copies, or parallel replicas, is $R = 3$. 
Solid dots and hollow 
circles correspond to trajectories 
at PDMP times $n \Delta t$. 
The 
cross indicates the 
terminal point of a $\Delta t$-time 
interval corresponding 
to an escape from $W$. Times of copies that 
contribute to $g_{par}$ 
and $T_{par}$ are pictured with 
solid line segments and solid dots, while 
times that do not contribute are pictured with dotted lines and hollow circles.
Left: The 
synchronous parallel 
step, Algorithm~\ref{alg_parstep3}. 
Copy 
$r = 2$ escapes at PDMP time $\tau^2$. 
In this example $N = 4$ and $J = 2$.
Right: The 
asynchronous parallel 
step, Algorithm~\ref{alg_parstep4}. 
Copy $r=2$ is the first to 
escape in terms of the wall-clock 
time ordering. The notches 
on the wall-clock time axis 
are the values of $t_{wall}^{r_k}(m_k)$, 
$k=1,\ldots,11$. In this 
example
$K=5$.}
\label{fig_parstep2}
\end{figure}

The speedup from the parallel 
step can be up to a factor of $R$, the number of copies or replicas~\cite{aristoff2014parallel,aristoff2015parallel, le2012mathematical,voter1998parallel,wang2018stationary}, 
provided the underlying PDMP 
simulation algorithm is based on time discretization. If the PDMP simulation algorithm is based on computing the skeleton chain, then 
the speedup in Algorithm~\ref{alg_parstep3} 
may be reduced. 
This can be mitigated, however,
by using Algorithm~\ref{alg_parstep4} instead.
The parallel step is consistent for any $W$, with a speedup if $W$ is metastable for the PDMP.

Theorem~\ref{thm_consistent2} 
gives conditions that establish 
consistency of the parallel steps,
Algorithm~\ref{alg_parstep3} and~\ref{alg_parstep4}. 
For the asynchronous parallel step, 
the crucial condition essentially 
says that the wall-clock time it takes to compute 
a $\Delta t$ time interval of $Z(t)_{t\ge 0}$ is independent of its position. This is reasonable if $Z(t)_{t\ge 0}$ is 
simulated via a time discretization  
technique with a fixed time step. It may not be reasonable 
if a skeleton chain-based technique, 
like Algorithm~\ref{alg1}, is used instead.

\begin{algorithm}
\caption{Continuous time computation of stationary averages}
Choose an initial point $Z(0)$, set $f_{sim} = 0$, 
$T_{sim} = 0$, and iterate:
\vskip5pt

{\bf 1.}  Starting at $Z(0)$, evolve $Z(t)_{t \ge 0}$ forward in time, stopping at time
\begin{equation*}
S = \inf\{t\ge T_{corr}^\mu(W):\exists\,W \in {\mathcal V}\,s.t.\,Z(s) \in W,\, s \in [t-T_{corr}^\mu(W),t]\},
\end{equation*}
the first time it remains in some $W \in {\mathcal V}$ for consecutive time
$T_{corr}^\mu(W)$. Set
\begin{equation*}
f_{decorr} =  \int_0^{S} f(Z(t))\,dt
\end{equation*}
and $T_{decorr} = {\mathbbm{1}}_{decorr}$ using the same formula. Store this $W$ for Step 2 and update $$f_{sim} \leftarrow f_{sim} + f_{decorr}, \qquad  T_{sim} \leftarrow T_{sim} + T_{decorr}.$$

{\bf 2.} Run the parallel step (Algorithm~\ref{alg_parstep3} or~\ref{alg_parstep4}) in the set $W$ from Step 1. Update $f_{sim} \leftarrow f_{sim} + f_{par}$, $T_{sim} \leftarrow T_{sim} + T_{par}$, set $Z(0) = Z_{par}$, and then go to Step 1. 
\vskip5pt

The algorithm stops when $T_{sim}$ exceeds a user-chosen threshold $T_{stop}$. At this time, 
\begin{equation*}
\langle f \rangle \approx \frac{f_{sim}}{T_{sim}}
\end{equation*}
is our estimate of the stationary average~\eqref{ergodic}.\label{alg_main_dyn}
\end{algorithm}

\begin{theorem}[Consistency of the parallel steps Algorithm~\ref{alg_parstep3} 
and~\ref{alg_parstep4}]\label{thm_consistent2}
\hskip150pt
\vskip2pt
\noindent (i) Let $\mu$ be the QSD of $Z(t)_{t \ge 0}$ in some $W \in {\mathcal V}$, 
suppose that ${\mathcal L}(Z(0)) = \mu$, 
and define $\tau = \inf\{t\ge 0:Z(t) \notin W\}$. Then in Algorithm~\ref{alg_parstep3},
\begin{equation}\label{avg_correct3}
{\mathbb E}(f_{par}) = {\mathbb E}\left(\int_0^\tau f(Z(t))\,dt\right)
\end{equation} 
and
\begin{equation}\label{exit_correct3}
 (T_{par},Z_{par}) \stackrel{\mathcal L}{=} (\tau,Z(\tau)).
\end{equation}
\vskip5pt
\noindent (ii) 
Suppose 
$(t_{wall}(m)_{m \ge 0})^{1 \le r \le R}$, the wall-clock times from Algorithm~\ref{alg_parstep4}, satisfy
the assumptions 
of Proposition~\ref{prop_asynch} when $(Y^r(t)_{t \ge 0})^{r=1,\ldots,R}$ 
equals $(Z^r(t)_{t \ge 0})^{r=1,\ldots,R}$. Adopt the 
assumptions in (i) above. Then~\eqref{avg_correct3}-\eqref{exit_correct3} hold.
\end{theorem}

Algorithm~\ref{alg_main_dyn} generates correct stationary 
averages by the same 
argument as in the previous 
section. 
It is worth mentioning that
Algorithms~\ref{alg_parstep3} 
and~\ref{alg_parstep4}
have a property not shared by 
Algorithms~\ref{alg_parstep1} and \ref{alg_parstep2}: the 
escape events $(T_{par},Z_{par})$ in these 
parallel steps
have the correct law for the PDMP. 
This allows us to use Algorithm~\ref{alg_main_dyn} to 
compute the dynamics of $Z(t)_{t\ge 0}$. 
More precisely, Algorithm~\ref{alg_main_dyn} 
leads to a PDMP dynamics 
that is correct on the 
quotient space obtained 
by considering each $W \in {\mathcal V}$ 
as a single point. Note 
that Algorithm~\ref{alg_main} 
cannot be used in this 
way, as it generates 
dynamics of the skeleton 
chain and not the PDMP.

\section{Numerics}\label{sec:numerics}

Here we test our algorithms 
above on a toy 
PDMP model, our aim 
being to illustrate  
Algorithms~\ref{alg_main} 
and~\ref{alg_main_dyn}. 
We will use these algorithms 
to sample the stationary 
average of a function $f$ with 
respect to the Boltzmann 
density
$\pi = Z^{-1}e^{-\beta V}$, 
where $f$ and $V$ are defined 
below and $\beta > 0$ 
is inverse temperature. The 
toy model is a two-dimensional 
version of a PDMP that may be 
defined in 
an arbitrary dimension $d-1$, as 
follows. Let $d_0,\ldots,d_{N-1} \in {\mathbb R}^{d-1}$ be direction vectors such that 
$d_0 + \ldots + d_{N-1} = 0$. Let 
${\mathbb Z}_N$ denote the integers 
modulo $N$, consider the indices of the $d_k$'s as elements of ${\mathbb Z}_N$, and for $k,\ell \in {\mathbb Z}_N$ 
define
\begin{align*}
F_{k,\ell}(x) &= \beta (d_k + \ldots d_{k+\ell})\cdot \nabla V(x).
\end{align*}

\begin{figure}
\includegraphics[width=7.5cm]{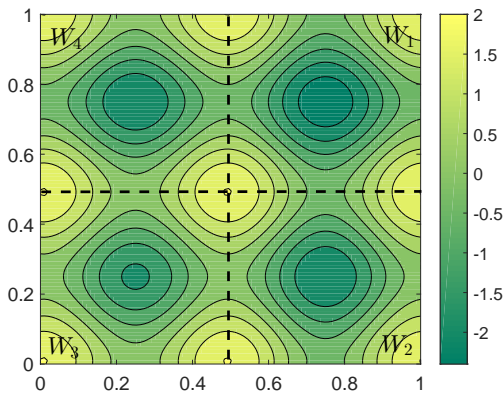}
\centering
\caption{Contour plot of the potential $V$ in~\eqref{Vpotential}.
$V$ has $4$ basins of attraction 
$W_1,W_2,W_3,W_4$ of  
different depths, with 
$W_1$ the deepest. Recall a basin of 
attraction for $V$ is a 
set of initial conditions $x(0)$ for which the differential 
equation $dx(t)/dt = -\nabla V(x(t))$ 
has a unique long-time limit.}
\label{fig0}
\end{figure}
Consider the PDMP with generator defined by
\begin{equation}\label{ell}
Lg(x,k) = d_k \cdot \nabla g(x,k) 
+ \left[g(x,k-1)-g(x,k)\right]\max_{0 \le \ell \le N-1}F_{k,\ell}(x)
\end{equation}
for suitable 
$g:\Omega \times {\mathbb Z}_N \to {\mathbb R}$ 
where either $\Omega = {\mathbb R}^{d-1}$ 
or $\Omega$ is a cube in ${\mathbb R}^{d-1}$
with periodic boundaries. This is the generator 
for a PDMP that, when moving in direction $d_k$ at point $x$, switches to direction 
$d_{k-1}$ with rate $\max_{0 \le \ell \le N-1}F_{k,\ell}(x)$. The 
resulting process can be seen 
as a rejection-free or ``lifted'' 
version of the sequential Metropolis 
algorithm~\cite{kapfer2017irreversible,metropolis1953equation}, historically the first 
nonreversible sampling algorithm~\cite{hastings1970monte,metropolis1953equation} for sampling the Boltzmann distribution. Straightforward calculations show this PDMP has invariant density proportional 
to $\pi$; 
see Remark~\ref{prop_balance1} 
in Section~\ref{sec:proofs2}.

We consider the case where 
state space is $\Omega = [0,1]^2$ with 
periodic boundaries, $N = 4$, 
$d_0 = (1,0),\,d_1 = (-1,0),\, 
d_2 = (0,1),\,d_3 = (0,-1)$, 
and
the potential energy
$V$ is
pictured in Figure~\ref{fig0}.
Specifically 
\begin{equation}\label{Vpotential}
    V(x,y) = \cos(4\pi x)+\cos(4\pi y)+\frac{1}{5}\sin(2\pi x)+\frac{1}{5}\sin(2\pi y).
\end{equation}
We define
${\mathcal W}$ and ${\mathcal V}$ 
using the basins of attraction 
$W_i$, $i=1,\ldots,4$ defined 
as
the four squares of 
equal side length $1/2$ inside $[0,1]^2$. See Figure~\ref{fig0}. 
Thus with $x$ and $k$ the position and direction variables, respectively, of the skeleton chain and PDMP, and $\theta$ the jump time variable of the skeleton chain,
\begin{align*}
{\mathcal W} &=\left\{\{(x,k,\theta):x \in W_i,\,k \in\{0,1,2,3\},\, \theta>0\}: i=1,2,3,4\right\},\\
{\mathcal V} &= \left\{\{(x,k):x \in W_i,\,k \in\{0,1,2,3\}\}:i=1,2,3,4\right\}.
\end{align*}
That is, the skeleton chain or PDMP is in a given set in ${\mathcal W}$ or 
${\mathcal V}$ at a 
particular time if and only if its position 
variable belongs to a given 
$W_i$ at that time.

\begin{algorithm}
\caption{Time discretization of the PDMP~\eqref{ell}}\label{alg_ell}
Choose an initial point $Z(0) \in {\mathbb R}^{d-1} \times\{0,\ldots,N-1\}$.  and 
pick $d_0,\ldots,d_{N-1}\in {\mathbb R}^{d-1}$ with 
$\sum_{k=0}^{N-1} d_k = 0$. Choose a time step $\delta t>0$. Then set 
$t = 0$ and iterate:
\vskip5pt
\noindent {\bf 1.} If $Z(t) = (x,k)$, 
define an acceptance probability
$$p = \min_{0 \le \ell \le N-1} \exp\left(\beta V(x)-\beta V(x+d_k\delta t+\ldots + d_{k+\ell}\delta t)\right).$$

\noindent {\bf 2.} With probability $p$, set $Z(t+\delta t) = (x+d_k\delta t,k)$, else set $Z(t+\delta t) = (x,k-1)$.

\noindent {\bf 3.} Update $t \leftarrow t+\delta t$ and return to Step 1.

\vskip5pt
Here, $Z(n\delta t)_{n \ge 0}$ has invariant measure proportional to $e^{-\beta V}$; see Remark~\ref{rmk_ell}.
\end{algorithm}

We tested Algorithm~\ref{alg_main} 
and~\ref{alg_main_dyn} with 
the synchronous parallel steps 
Algorithm~\ref{alg_parstep1} 
and Algorithm~\ref{alg_parstep3}, 
respectively. We used 
both algorithms to estimate the stationary 
average
$\langle f \rangle$ where $f(x,k) = {\mathbbm{1}}_{x \in W_1}$, 
the characteristic function of the deepest basin of $V$.
We used up to $R = 100$ replicas and decorrelation 
times that 
were the same in each basin, $T_{corr}^\nu \equiv T_{corr}^\nu(W_i)$ and $T_{corr}^\mu \equiv T_{corr}^\mu(W_i)$, $i=1,\ldots,4$. We 
used Algorithm~\ref{alg_ell}
 with time step $\delta t=10^{-2}$
to simulate the PDMP. 
In Algorithm~\ref{alg_main_dyn} we took 
$\Delta t = \delta t =10^{-2}$.
The results 
are in Figures~\ref{fig1},~\ref{fig2}
and~\ref{fig3}.

\begin{figure}
\subfigure{\includegraphics[width=6.4cm]{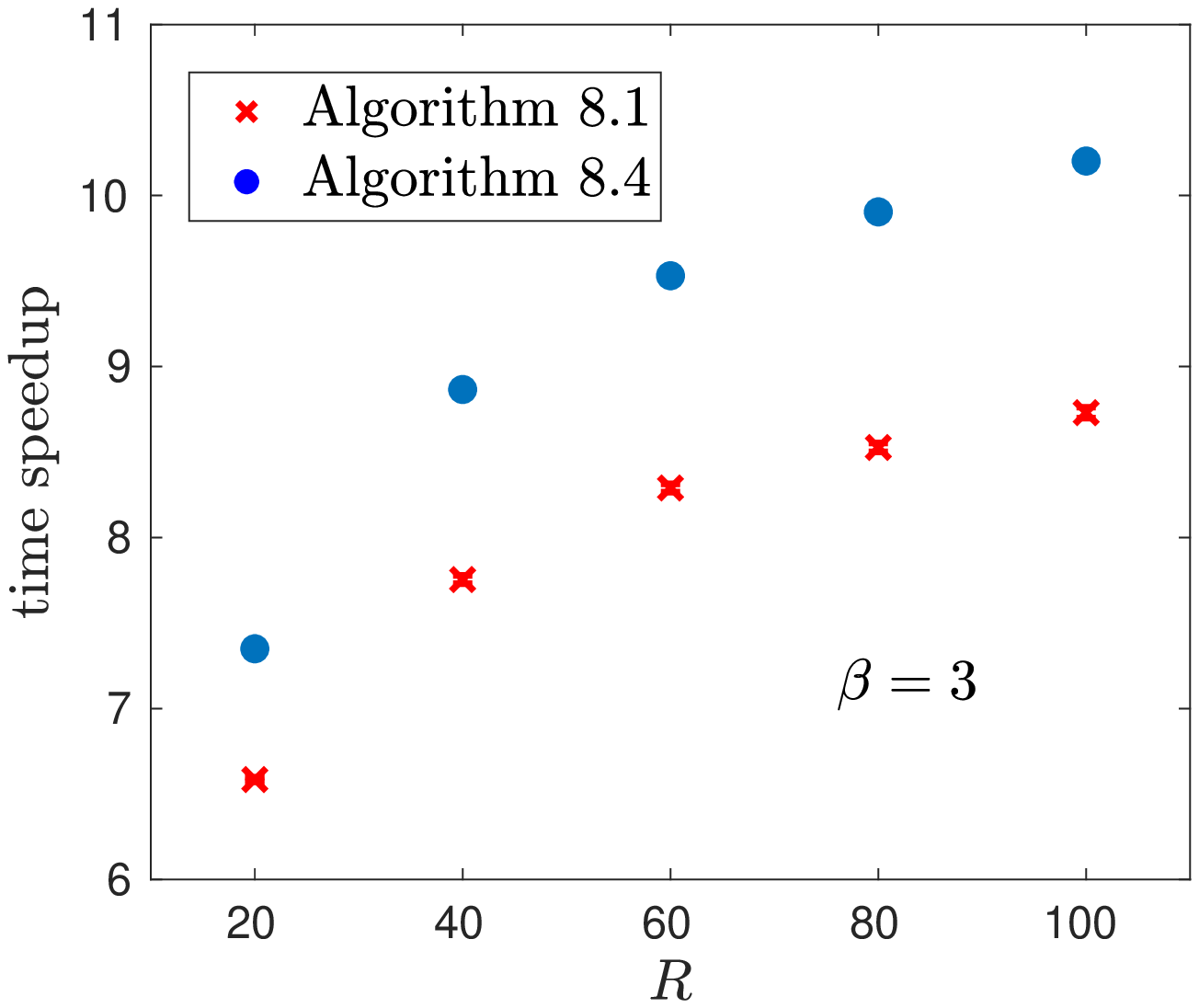}}
\subfigure{\includegraphics[width=6.4cm]{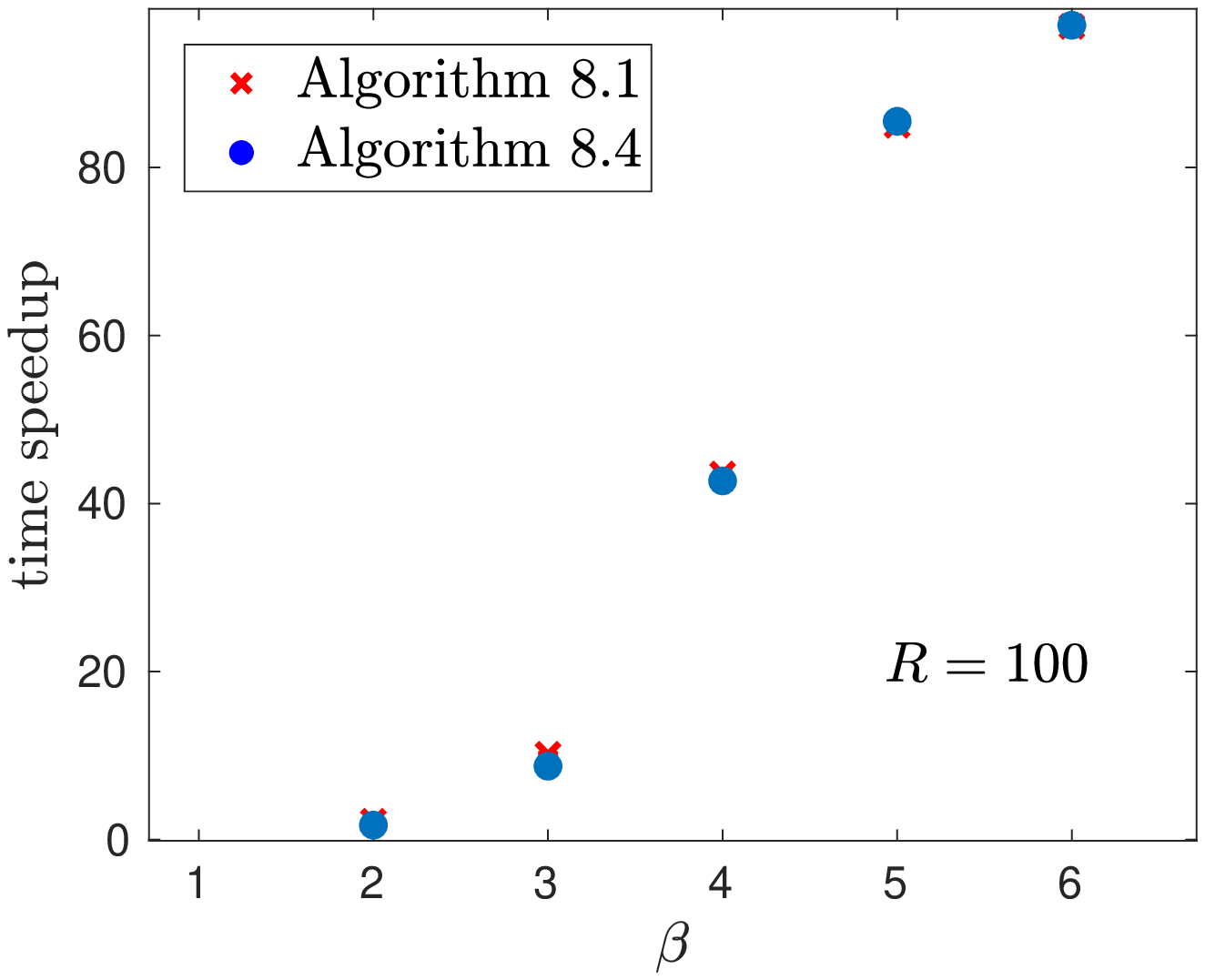}}
\caption{Left: Time speedup vs. number 
of replicas, $R$, 
when $\beta = 3$.
Right: 
Time speedup vs. $\beta$ when $R = 100$. In Algorithm~\ref{alg_parstep1} 
we used $T_{corr}^\nu = 100$, 
while in Algorithm~\ref{alg_parstep3} we 
used $T_{corr}^\mu = 6$. The decorrelation 
times were chosen so that 
Algorithm~\ref{alg_parstep1} and~\ref{alg_parstep3} would 
have similar values for the time 
speedup. Error bars for each 
data point were obtained from $50$ independent 
simulations, but they are smaller than 
the data markers. In the limit $\beta \to \infty$, the computational effort to sample 
the QSD vanishes compared to the 
effort to generate an escape 
event using serial simulation.
In this limit, at left, 
we expect scaling 
like $R$ for finite $R$. 
On the other hand, at right, we expect
the time speedup 
to level off at $R = 100$ 
as $\beta \to \infty$. This figure 
is based on simulations 
performed by Peter Christman.}
\label{fig1}
\end{figure}

To analyze our results, we 
defined an idealized speedup factor as follows. 
Let $T^{R}$ be an idealized 
wall-clock time for a 
simulation of Algorithm~\ref{alg_main} or~\ref{alg_main_dyn} using the parallel steps 
Algorithm~\ref{alg_parstep1} 
and~\ref{alg_parstep3}, respectively, up to a 
fixed time $T_{stop}$. The 
idealized wall-clock time $T^R$ is
obtained by  
assuming that we use $R$ parallel processors with zero 
communication cost, such that
on each processor, one 
step of the skeleton chain 
is computed in wall clock 
time~$1$. Writing $T^1$ for the wall-clock time corresponding to  
$1$ processor or direct serial simulation, we define
\begin{equation*}
\text{time speedup} = \frac{T^{R}}{T^{1}}.
\end{equation*}
We include computation time from the 
dephasing step -- {\em i.e.} the QSD sampling 
step in Algorithms~\ref{alg_parstep1} 
and~\ref{alg_parstep3} -- as 
part of $T^{R}$. We assume 
this dephasing is done using 
a Fleming-Viot-based technique 
as described above, with $R$ copies 
of the underlying skeleton chain or PDMP.

Processor communication, 
which we do not account 
for, 
of course takes a toll on the time 
speedup. However, 
the processor communication cost 
is small compared 
to the rest of the computational effort
if the sets in ${\mathcal W}$ 
and ${\mathcal V}$ are 
significantly metastable. Thus, our time speedup 
gives a reasonable picture of 
the gain that can be expected.

The time speedup depends on the parameters in Algorithms~\ref{alg_main} and~\ref{alg_main_dyn}. 
If all other parameters are held constant,
the time speedup 
increases with $R$ or $\beta$, due 
to increasing parallelization or metastability, 
respectively (Figure~\ref{fig1}), while the 
time speedup decreases 
with $T_{corr}^\nu$ and 
$T_{corr}^\mu$, due to increased 
effort to sample the QSD (Figure~\ref{fig2}). 
With increasing metastability, 
the relative computational effort to sample 
the QSD decreases in comparison with 
the effort to simulate an 
escape from a metastable set. 
Since the latter is done in 
parallel, increasing the metastability 
leads to a larger time speedup.

\begin{figure}
\subfigure{\includegraphics[width=6.45cm]{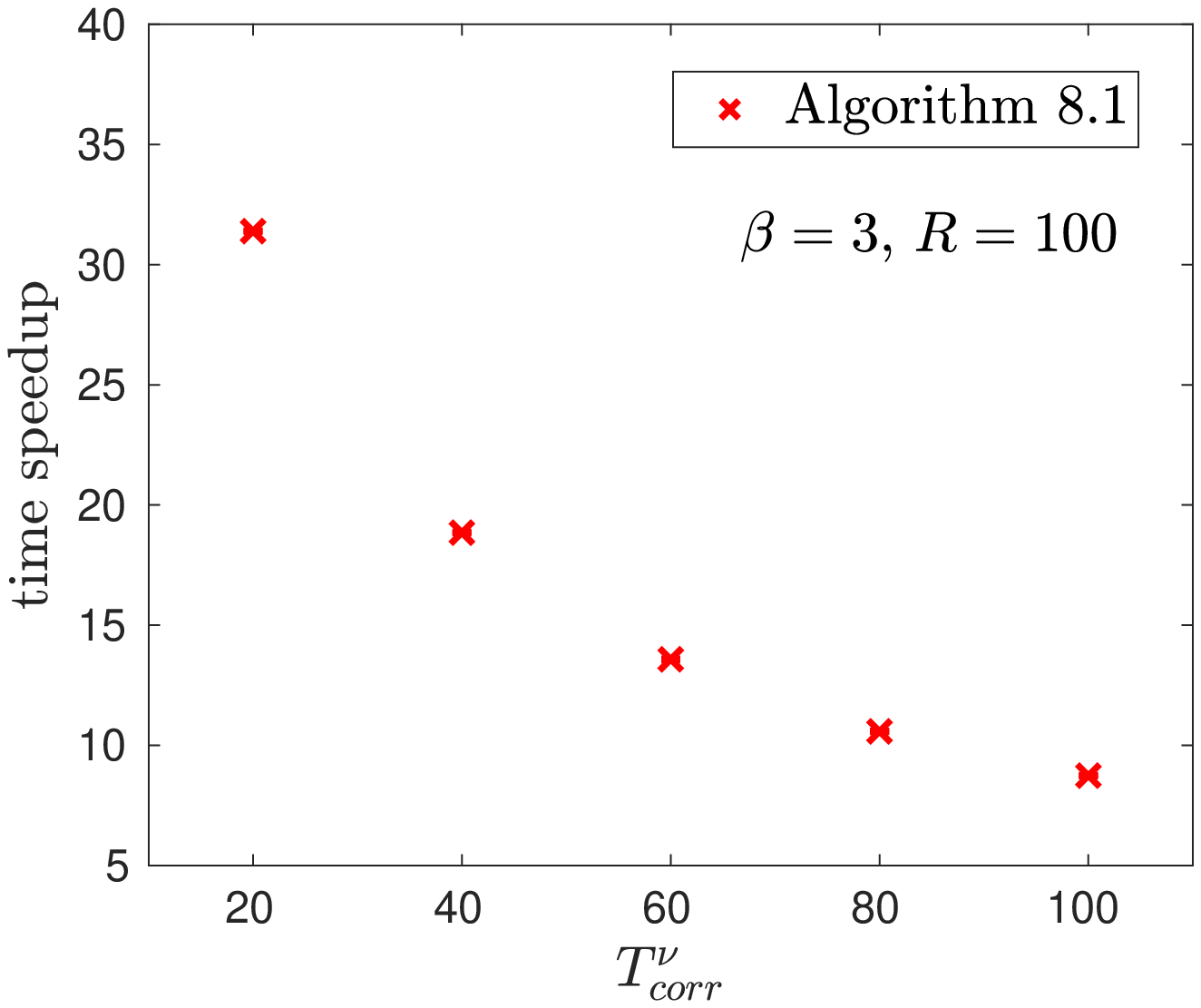}}
\subfigure{\includegraphics[width=6.45cm]{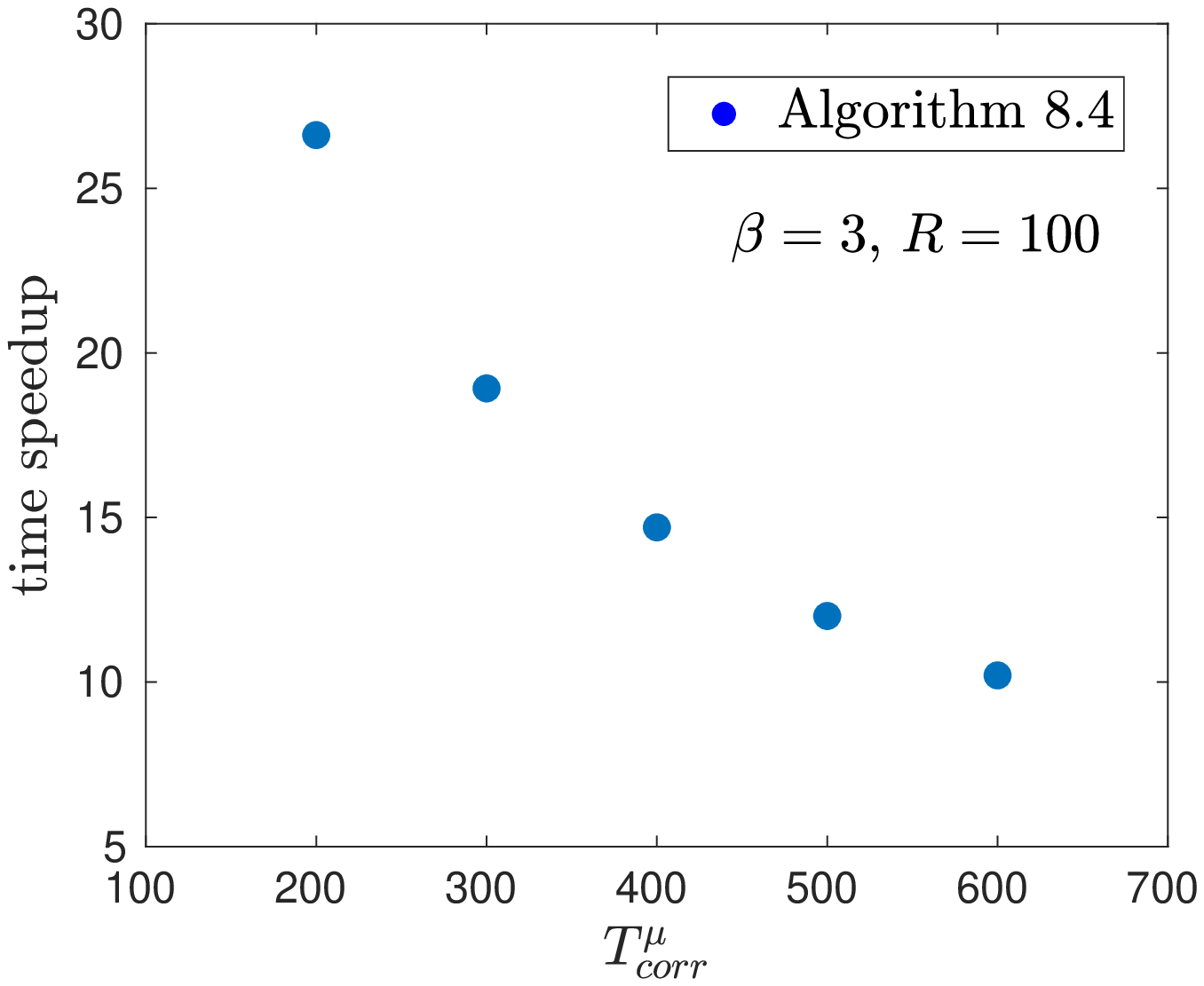}}

\caption{Left: Time speedup vs. $T_{corr}^\nu$
in Algorithm~\ref{alg_parstep1}. 
Right: Time speedup vs. $T_{corr}^\mu$ 
in Algorithm~\ref{alg_parstep3}. 
In both plots $\beta = 3$ and $R = 100$. 
Error bars for each data point were obtained from $50$ independent 
simulations, but they are smaller than the 
data markers. This figure is 
based on simulations performed 
by Peter Christman.}
\label{fig2}
\end{figure}

\begin{figure}
\subfigure{\includegraphics[width=6.45cm]{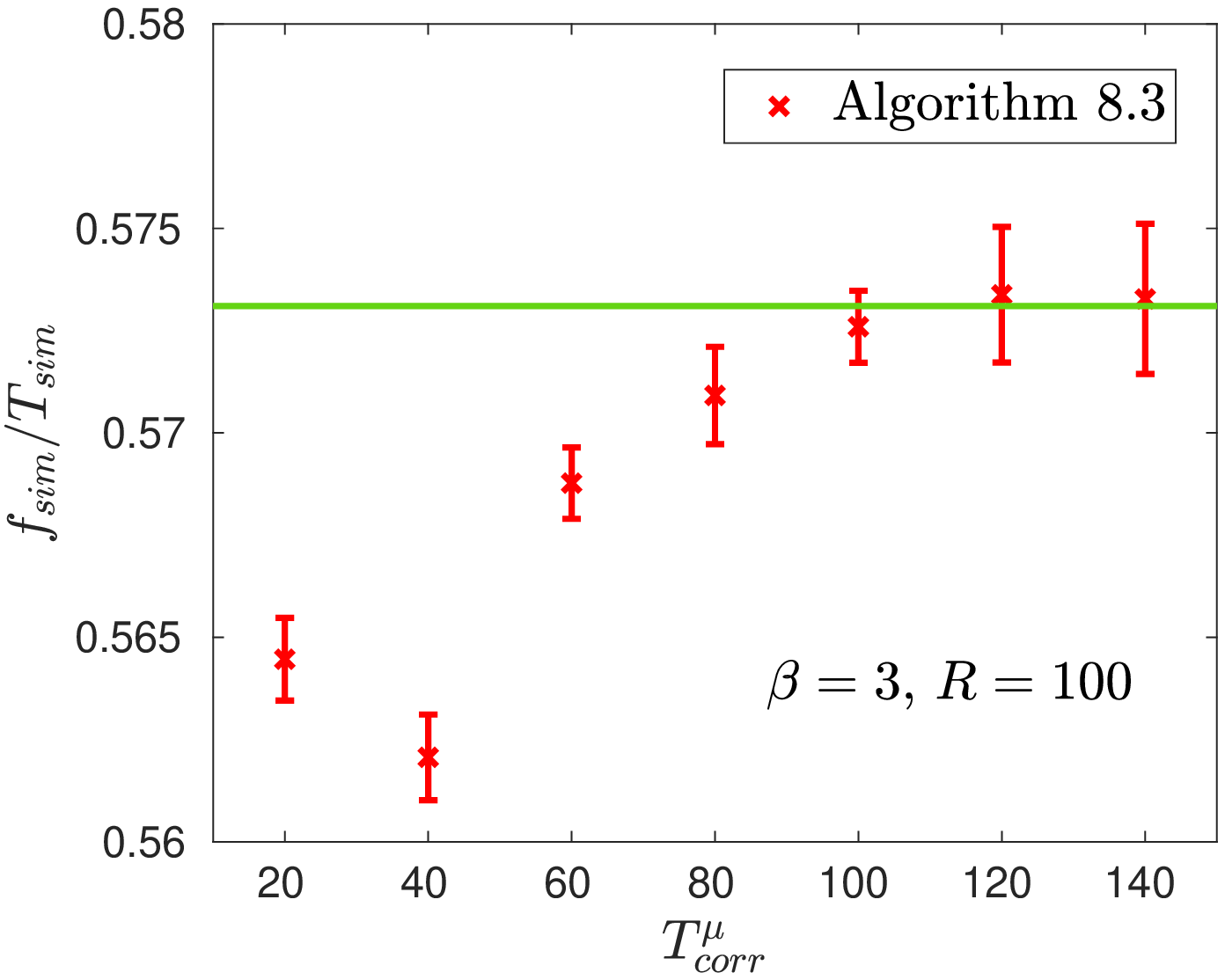}}
\subfigure{\includegraphics[width=6.45cm]{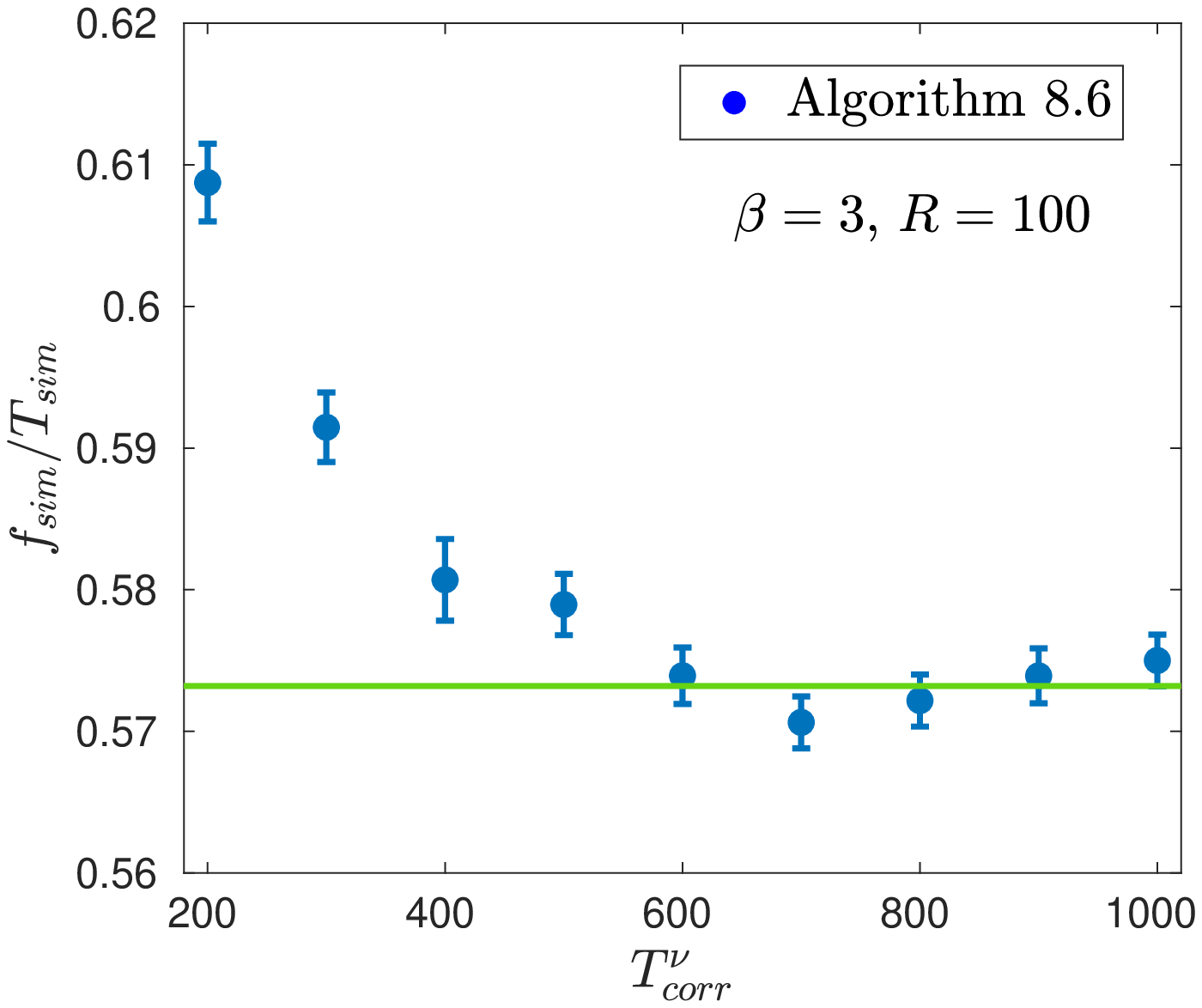}}

\caption{Left: Approximation $f_{sim}/T_{sim}$ of the 
stationary average $\langle f \rangle$ using 
Algorithm~\ref{alg_main}. Right: 
Approximation $f_{sim}/T_{sim}$ of the 
stationary average $\langle f \rangle$ using 
Algorithm~\ref{alg_main_dyn}. In both plots 
$\beta = 3$ and $R = 100$, and 
simulations ran for $T_{sim}$ exceeding 
$10^6$. 
Error bars are 
empirical standard deviations obtained from $50$ independent simulations for each data point. 
The exact 
value $\langle f \rangle$ is indicated 
with a solid line. This figure is 
based on simulations performed 
by Peter Christman.}
\label{fig3}
\end{figure}

Figure~\ref{fig3} shows 
that the approximation 
$f_{sim}/T_{sim}$ approaches 
the stationary average 
$\langle f\rangle$ as the 
decorrelation times
$T_{corr}^\nu$ and $T_{corr}^\mu$ 
increase, as expected. As 
discussed above, appropriate 
values of these QSD sampling times 
depend on the degree 
of metastability.
Note that the approximations
$f_{sim}/T_{sim}$ are quite 
good even for small QSD sampling times. 
This was true not just for $f(x,k) = {\mathbbm{1}}_{x \in W_1}$ 
but for a
variety of other functions. This 
feature, of reasonable accuracy in ParRep
even for relatively small decorrelation 
times, was observed before 
in~\cite{wang2017parallel,wang2018stationary}. 
Here this may be a result of the momentum-like direction variables, which 
can make the PDMP unlikely to immediately 
escape from a metastable set just 
after entering.

\section{Proofs}\label{sec:proofs}

Our first two results 
below, Proposition~\ref{prop_exp_exit} and~\ref{prop_indep_exit}, 
establish the memoryless 
distribution of the escape 
time starting from the QSD, and the independence of 
the escape time and escape 
point, for general 
Markov processes. See for instance~\cite{collet2012quasi} for details. 
We include proofs 
for completeness.

\begin{proposition}\label{prop_exp_exit}
Let $X(t)_{t \ge 0}$ have a
QSD $\rho$ in $U$, and suppose 
${\mathcal L}(X(0)) = \rho$. Suppose
$T= \inf\{t\ge 0: X(t) \notin U\}$ 
is finite almost surely. 
Then $T$
has a memoryless distribution; 
that is, $T$ is either  
exponentially or geometrically distributed.
\end{proposition}

\begin{proof}
The definition~\eqref{QSDdef} of the QSD 
together with the Markov property show that
${\mathcal L}(X(t+s)_{s \ge 0} | T > t) = {\mathcal L}(X(s)_{s \ge 0})$.
Thus, ${\mathbb P}(T > t+s|T> t) = {\mathbb P}(T > s)$ for any $s,t\ge 0$.
When $T$ is finite-valued, the only distribution satisfying this 
is
\begin{equation}\label{expgeom}
{\mathbb P}(T>t) = e^{-\lambda t}.
\end{equation}
We use~\eqref{expgeom} to indicate either the (continuous) exponential distribution 
with parameter $\lambda > 0$, 
or the (discrete) geometric distribution 
with parameter $p = 1-e^{-\lambda}$.
\end{proof}
 
 \begin{proposition}\label{prop_indep_exit}
 Let $X(t)_{t \ge 0}$ have a 
 QSD $\rho$ in $U$, and suppose 
 ${\mathcal L}(X(0)) = \rho$. 
Suppose $T = \inf\{t \ge 0:X(t) \notin U\}$ 
is finite almost surely. Then 
 $T$ and $X(T)$ are 
 independent.
 \end{proposition}
 
 \begin{proof}
 By Proposition~\ref{prop_exp_exit}, ${\mathbb P}(T>t) = e^{-\lambda t}$. For $A$ in the complement of $U$,
 \begin{align*}
{\mathbb P}(T \in ((n-1)t,nt],X(T) \in A) 
&= {\mathbb P}(T > (n-1)t, 
X(nt \wedge T) \in A) \\
&= {\mathbb P}(X(nt\wedge T) \in A|T>(n-1)t){\mathbb P}(T>(n-1)t)\\
&= {\mathbb P}(X(T) \in A,T \le nt|T>(n-1)t){\mathbb P}(T>(n-1)t)\\
&= {\mathbb P}(X(T)\in A,T \le t)e^{-\lambda (n-1)t},
\end{align*}
where the last step uses~\eqref{QSDdef}. 
Summing over $n \ge 1$ establishes the result:
\begin{align*}
{\mathbb P}(X(T) \in A) 
&= {\mathbb P}(X(T) \in A, T\le t)\frac{1}{1-e^{-\lambda t}} \\
&= \frac{{\mathbb P}(X(T) \in A, T\le t)}{{\mathbb P}(T \le t)}.
\end{align*}

 \end{proof}
 
Below, we write MGF 
for the moment generating 
function of a random variable.
The results in 
Proposition~\ref{prop_1} below 
hold in both continuous and 
discrete time. To connect
the discrete and continuous time cases, 
we write $1-p = e^{-\lambda}$ 
where $p \in (0,1)$ is the geometric parameter 
and $\lambda>0$ is the exponential rate. 

\begin{proposition}\label{prop_1}
Let 
$t_1,t_2,\ldots$ be nonnegative deterministic times such that $\sum_{m=1}^\infty t_m = \infty$. Let $\tau_1,\tau_2,\ldots$ be 
random variables such that ${\mathbb P}(\tau_1>t) =  e^{-\lambda t}$ and 
\begin{equation}\label{recurse}{\mathbb P}(\tau_m>t|\tau_{m-1}>t_{m-1},\ldots,\tau_1 >t_1) = e^{-\lambda t}, \quad m\ge 2.
\end{equation} 
Let 
$L= \inf\{m \ge 1:\tau_m \le t_m\}$.
Then $${\mathbb P}(t_1+\ldots + t_{L-1}+\tau_L>t)= e^{-\lambda t}.$$
\end{proposition}

\begin{proof}
Let $s_m = t_1+\ldots +t_m$ and $s_0 = 0$. 
By~\eqref{recurse} and induction,
\begin{equation}\label{induction}
{\mathbb P}(\tau_{m}>t_{m},\ldots,\tau_1>t_1)= e^{-\lambda (t_1+\ldots + t_{m})} = e^{-\lambda s_{m}}.
\end{equation}
Using~\eqref{recurse} again, in the continuous case,
\begin{align}\begin{split}\label{MGF1}
&{\mathbb E}\left(e^{u\tau_m}\mathbbm{1}_{\tau_m \le t_m}|\tau_{m-1}>t_{m-1},\ldots,\tau_1>t_1\right) \\
&= \int_0^{t_m} e^{us}\lambda e^{-\lambda s}\,ds =  \frac{\lambda}{u-\lambda} \left(e^{(u-\lambda)t_m}-1\right),
\end{split}
\end{align}
while in the discrete case, where $e^{-\lambda} = 1-p$, 
\begin{align}\begin{split}\label{MGF2}
&{\mathbb E}\left(e^{u\tau_m}\mathbbm{1}_{\tau_m \le t_m}|\tau_{m-1}>t_{m-1},\ldots,\tau_1>t_1\right) \\
&= \sum_{s=1}^{t_m} e^{us}e^{-\lambda (s-1)}(1-e^{-\lambda}) = \frac{e^u-e^{u-\lambda}}{1-e^{u-\lambda}}\left(1-e^{(u-\lambda)t_m}\right).
\end{split}
\end{align}
Note also that 
\begin{equation}\label{LM}
\{L = m\} = \{\tau_m \le t_m,\tau_{m-1} > t_{m-1},\ldots,\tau_{1} > t_1\}.
\end{equation}
Consider the continuous case. Combining~\eqref{induction},~\eqref{MGF1} and~\eqref{LM} gives
\begin{align}\begin{split}\label{uzero}
&{\mathbb E}\left(e^{u\tau_m}\mathbbm{1}_{L=m}\right)\\ &= {\mathbb E}\left(e^{u\tau_m}\mathbbm{1}_{\tau_m\le t_m}\mathbbm{1}_{\tau_{m-1}>t_{m-1},\ldots,\tau_1>t_1}\right)\\
 &= {\mathbb E}\left(e^{u\tau_m}\mathbbm{1}_{\tau_m\le t_m}|{\tau_{m-1}>t_{m-1},\ldots,\tau_1>t_1}
 \right){\mathbb P}(\tau_{m-1}>t_{m-1},\ldots,\tau_1>t_1)\\
&= \frac{\lambda}{u-\lambda} \left(e^{(u-\lambda)t_m}-1\right) e^{-\lambda s_{m-1}}.
\end{split}
\end{align}
We now see that $s_{L-1}+\tau_L$ has the MGF of an exponential($\lambda$) random variable:
\begin{align*}
{\mathbb E}\left(e^{u(s_{L-1}+\tau_L)}\right) &= \sum_{m=1}^\infty {\mathbb E}\left(e^{u(s_{L-1}+\tau_L)}\mathbbm{1}_{L=m}\right)  \\
&= \sum_{m=1}^\infty e^{us_{m-1}}{\mathbb E}\left(e^{u\tau_m}\mathbbm{1}_{L=m}\right)\\
&= \frac{\lambda}{u-\lambda}\sum_{m=1}^\infty \left(e^{(u-\lambda)s_m}-e^{(u-\lambda)s_{m-1}}\right)\\
&= \frac{\lambda}{\lambda - u} \qquad \text{if }u<\lambda.
\end{align*}
Similarly, in the discrete case, combining~\eqref{induction},~\eqref{MGF2} and~\eqref{LM} gives
\begin{align}\begin{split}\label{uzero2}
&{\mathbb E}\left(e^{u\tau_m}\mathbbm{1}_{L=m}\right)\\ &= {\mathbb E}\left(e^{u\tau_m}\mathbbm{1}_{\tau_m\le t_m}\mathbbm{1}_{\tau_{m-1}>t_{m-1},\ldots,\tau_1>t_1}\right)\\
&= {\mathbb E}\left(e^{u\tau_m}\mathbbm{1}_{\tau_m\le t_m}|{\tau_{m-1}>t_{m-1},\ldots,\tau_1>t_1}\right){\mathbb P}(\tau_{m-1}>t_{m-1},\ldots,\tau_1>t_1)\\
&= \frac{e^u-e^{u-\lambda}}{1-e^{u-\lambda}}\left(1-e^{(u-\lambda)t_m}\right) e^{-\lambda s_{m-1}}.
\end{split}
\end{align}
This shows again that $s_{L-1}+\tau_L$ has the MGF of a geometric($p$) random variable, via
\begin{align*}
{\mathbb E}\left(e^{u(s_{L-1}+\tau_L)}\right) &= \sum_{m=1}^\infty {\mathbb E}\left(e^{u(s_{L-1}+\tau_L)}\mathbbm{1}_{L=m}\right)  \\
&= \sum_{m=1}^\infty e^{us_{m-1}}{\mathbb E}\left(e^{u\tau_m}\mathbbm{1}_{L=m}\right)\\
&= \frac{e^u-e^{u-\lambda}}{1-e^{u-\lambda}}\sum_{m=1}^\infty (e^{(u-\lambda)s_{m-1}}-e^{(u-\lambda)s_m}) \\
&=\frac{e^u-e^{u-\lambda}}{1-e^{u-\lambda}} 
\qquad \text{if }u<\lambda\\
&= \frac{pe^u}{1-(1-p)e^u} \qquad \text{ if }u< -\log(1-p).
\end{align*}
\end{proof}

\begin{proof}[Proof of Theorem~\ref{theorem_dyn}]
Let $A$ be a subset of the complement of $U$. Note that, due to~\eqref{assump1}, 
the events $\{X_m(T_m) \in A,T_m\le t\}$ and
$\{T_{m-1}>t_{m-1},\ldots,T_1>t_1\}$ are independent conditional on 
$\{X_m(0)= x\}$. Using this,~\eqref{assump2} and 
Proposition~\ref{prop_indep_exit},
\begin{align}\begin{split}\label{longprob}
&{\mathbb P}(X_m(T_m) \in A,T_m \le t|T_{m-1} > t_{m-1},\ldots,T_{1} > t_1) \\
&= \int {\mathbb P}(X_m(T_m) \in A,T_m \le  t|X_m(0) = x,T_{m-1} > t_{m-1},\ldots,T_{1} > t_1)\\
&\qquad \qquad \qquad \times
{\mathbb P}(X_m(0) \in dx|T_{m-1} > t_{m-1},\ldots,T_{1} > t_1) \\
&= \int {\mathbb P}(X_m(T_m) \in A,T_m \le  t|X_m(0) = x)\rho(dx) \\
&= {\mathbb P}(X(T) \in A,T \le  t) \\
&= {\mathbb P}(X(T) \in A){\mathbb P}(T \le t).\end{split}
\end{align}
Taking $A$ as the complement of $U$ in~\eqref{longprob}, and using Proposition~\ref{prop_exp_exit}, 
\begin{equation}\label{PTm}
    {\mathbb P}(T_m > t|T_{m-1} > t_{m-1},\ldots,T_{1} > t_1) =  {\mathbb P}(T>t) = e^{-\lambda t}, \quad \text{some }\lambda>0.
\end{equation}
By~\eqref{assump2}, ${\mathbb P}(T_1>t) = e^{-\lambda t}$. 
Thus by Proposition~\ref{prop_1}, ${\mathcal L}(T_{par}) = {\mathcal L}(T)$. Notice 
\begin{equation}\label{LM2}
\{L = m\} = \{T_m \le t_m,T_{m-1}>t_{m-1},\ldots,T_1>t_1\}.
\end{equation}
From~\eqref{longprob},~\eqref{PTm} and~\eqref{LM2} we have
\begin{align}\begin{split}\label{long2}
&{\mathbb P}(X_{par} \in A, T_{par} > t|L=m) \\
&={\mathbb P}(X_m(T_m) \in A, s_{m-1}+T_m > t|L=m)\\
&={\mathbb P}\left(\left.X_m(T_m) \in A, T_m > t-s_{m-1}\right|T_m \le t_m,T_{m-1} > t_{m-1},\ldots,T_{1} > t_1\right)\\
&=  {\mathbb P}\left(\left.X_m(T_m) \in A,T_m\in (t-s_{m-1},t_m] \right|T_{m-1} > t_{m-1},\ldots,T_{1} > t_1\right)\\
&\qquad \qquad \qquad \times
{\mathbb P}(T_m \le t_m|T_{m-1} > t_{m-1},\ldots,T_{1} > t_1)^{-1} \\
 &={\mathbb P}(X(T) \in A){\mathbb P}(T \in(t-s_{m-1},t_m]){\mathbb P}(T \le t_m)^{-1} \\
 &= {\mathbb P}(X(T) \in A){\mathbb P}(T>t-s_{m-1}|T\le t_m).\end{split}
\end{align} 
From~\eqref{long2} we conclude
\begin{align}\begin{split}\label{totprob}
{\mathbb P}(X_{par} \in A, T_{par} > t) 
&= \sum_{m=1}^\infty {\mathbb P}(X_{par} \in A,T_{par}>t|L=m){\mathbb P}(L=m)\\
&= {\mathbb P}(X(T) \in A)\sum_{m=1}^\infty {\mathbb P}(T > t-s_{m-1}|T\le t_m){\mathbb P}(L=m).
\end{split}
\end{align}
In the last display, taking $t =0$ shows ${\mathbb P}(X_{par} \in A)= {\mathbb P}(X(T) \in A)$, while taking $A$ as the complement of $U$ shows ${\mathbb P}(T_{par}>t) = \sum_{m=1}^\infty {\mathbb P}(T > t-s_{m-1}|T\le t_m){\mathbb P}(L=m)$. Thus,~\eqref{totprob} 
shows that ${\mathcal L}(X_{par}) = {\mathcal L}(X(T))$ 
and $X_{par}, T_{par}$ are independent. 
\end{proof}

\begin{proof}[Proof of Theorem~\ref{theorem_avg}]
We consider only the discrete time case, since 
the arguments in the continuous time case are 
analogous. 
Let $r \in {\mathbb N}$ be fixed. Observe that
\begin{equation*}
\sum_{t=0}^{r-1} {\mathbb P}(T>t) = 
\sum_{t=0}^{r-1} {\mathbb E}({\mathbbm{1}}_{T>t}) 
= \sum_{t=0}^{\infty} {\mathbb E}({\mathbbm{1}}_{T\wedge r>t}) 
= {\mathbb E}(T \wedge r).
\end{equation*}
By the preceding display and the 
the definition~\eqref{QSDdef} of 
the QSD $\rho$,
\begin{align}\begin{split}\label{eqabove}
{\mathbb E}\left(\sum_{t=0}^{T\wedge r-1}g(X(t))\right) &= 
\sum_{s=1}^\infty \sum_{t=0}^{s-1} {\mathbb E}\left(g(X(t)){\mathbbm{1}}_{T \wedge r = s}\right) \\
&= \sum_{t=0}^\infty {\mathbb E}\left(g(X(t)){\mathbbm{1}}_{T \wedge r > t}\right) \\
&= \sum_{t=0}^{r-1} {\mathbb E}\left(g(X(t)){\mathbbm{1}}_{T>t}\right) \\
&= \sum_{t=0}^{r-1} {\mathbb E}(g(X(t))|T>t){\mathbb P}(T>t)\\
&= \left(\int g\,d\rho\right){\mathbb E}(T \wedge r).\end{split}
\end{align}
Since 
$\{L \ge m\} = \{T_{m-1}>t_{m-1},\ldots,T_1>t_1\}$,
using~\eqref{assump1},~\eqref{assump2} and~\eqref{eqabove} we get 
\begin{align}\begin{split}\label{eqabove2}
&{\mathbb E}\left(\left.\sum_{t=0}^{T_m \wedge t_m-1} g(X_m(t))\right|L \ge m\right) \\
&= \int {\mathbb E}\left(\left.\sum_{t=0}^{T_m \wedge t_m-1} g(X_m(t))\right|X_m(0) = x,L\ge m\right) {\mathbb P}(X_m(0) \in dx|L \ge m)\\
&= \int {\mathbb E}\left(\left.\sum_{t=0}^{T_m \wedge t_m-1} g(X_m(t))\right|X_m(0) = x\right)\rho(dx)\\
&= {\mathbb E}\left(\sum_{t=0}^{T \wedge t_m-1} g(X(t))\right) = \left(\int g\,d\rho\right){\mathbb E}(T\wedge t_m)\\
&= \left(\int g\,d\rho\right){\mathbb E}(T_m \wedge t_m|L \ge m)
\end{split}
\end{align}
where the last line of~\eqref{eqabove2} follows from taking $g \equiv \mathbbm{1}$ in the first four lines of~\eqref{eqabove2}. By Theorem~\ref{theorem_dyn}, ${\mathcal L}(T) = {\mathcal L}(T_{par})=  {\mathcal L}\left(\sum_{m=1}^L T_m \wedge t_m\right)$ and thus
\begin{align}\begin{split}\label{eqabove3}
{\mathbb E}(T) &= {\mathbb E}\left(\sum_{m=1}^L T_m \wedge t_m\right)\\
&= \sum_{n=1}^\infty \sum_{m=1}^n{\mathbb E}\left(T_m \wedge t_m{\mathbbm{1}}_{L=n}\right) \\
&= \sum_{m=1}^\infty {\mathbb E}\left(T_m \wedge t_m {\mathbbm{1}}_{L \ge m}\right) \\
&= \sum_{m=1}^\infty {\mathbb E}(T_m \wedge t_m |L \ge m){\mathbb P}(L \ge m).
\end{split}
\end{align}
Now by~\eqref{eqabove2} and~\eqref{eqabove3},
\begin{align}\begin{split}\label{same}
{\mathbb E}\left(\sum_{m=1}^{L} \sum_{t=0}^{T_m \wedge t_m-1} g(X_m(t))\right) &= 
\sum_{n=1}^\infty \sum_{m=1}^{n} {\mathbb E}\left({\mathbbm{1}}_{L=n}\sum_{t=0}^{T_m \wedge t_m-1} g(X_m(t))\right) \\
&=  \sum_{m=1}^\infty  {\mathbb E}\left({\mathbbm{1}}_{L\ge m}\sum_{t=0}^{T_m \wedge t_m-1} g(X_m(t))\right) \\
&= \sum_{m=1}^\infty  {\mathbb E}\left(\left.\sum_{t=0}^{T_m \wedge t_m-1} g(X_m(t))\right|L\ge m\right){\mathbb P}(L\ge m) \\
&= \left(\int g\,d\rho\right)\sum_{m=1}^\infty {\mathbb E}(T_m \wedge t_m|L \ge m){\mathbb P}(L\ge m) \\
&= \left(\int g\,d\rho\right){\mathbb E}(T).\end{split}
\end{align}
Letting $r \to \infty$ in~\eqref{eqabove}, using dominated convergence, and 
comparing with~\eqref{same}, $${\mathbb E}\left(\sum_{m=1}^{L} \sum_{t=0}^{T_m \wedge t_m-1} g(X_m(t))\right) = {\mathbb E}\left(\sum_{t=0}^{T-1}g(X(t))\right)$$
as desired.
\end{proof}

Below we will need 
the following basic facts. 
\begin{lemma}\label{lem_basic}
Let ${\mathcal F}$, 
${\mathcal G}$,
${\mathcal H}$, and ${\mathcal K}$ be $\sigma$-algebras.
\begin{itemize}
    \item[] {(i)} Let $\sigma({\mathcal G},{\mathcal H})$ be the $\sigma$-algebra 
    generated by ${\mathcal G}$ and ${\mathcal H}$. Suppose that 
    ${\mathcal F}$,
    $\sigma({\mathcal G},{\mathcal H})$ 
    are independent conditional on ${\mathcal K}$, and that ${\mathcal G}$, ${\mathcal H}$ are independent conditional on ${\mathcal K}$. 
    Then ${\mathcal F}$, ${\mathcal G}$, ${\mathcal H}$ are mutually 
    independent conditional on ${\mathcal K}$.
    \item[] {(ii)} Suppose ${\mathcal H}\subseteq {\mathcal G}$. If ${\mathcal F}$ 
    and ${\mathcal G}$ are independent, 
    then ${\mathcal F}$ 
    and ${\mathcal G}$ are independent 
    conditional on ${\mathcal H}$.
\end{itemize}
\end{lemma}
\begin{proof}
Throughout let $A \in {\mathcal F}$, 
$B \in {\mathcal G}$, and 
$C \in {\mathcal H}$, and write 
    ${\mathbbm{1}}_S$ for the 
    characteristic or indicator function of a set $S$. Consider {\em (i)}. Since $A \in {\mathcal F}$, $B \cap C \in \sigma({\mathcal G},{\mathcal H})$, and ${\mathcal F}$,  $\sigma({\mathcal G},{\mathcal H})$
 are independent conditional on ${\mathcal K}$, ${\mathbb E}(\mathbbm{1}_A \mathbbm{1}_{B \cap C}|{\mathcal K}) 
    = 
    {\mathbb E}(\mathbbm{1}_A|{\mathcal K}) {\mathbb E}(\mathbbm{1}_{B \cap C}|{\mathcal K})$ almost surely. Similarly $ {\mathbb E}(\mathbbm{1}_{B}\mathbbm{1}_C|{\mathcal K})
    = {\mathbb E}(\mathbbm{1}_{B}|{\mathcal K}) 
    {\mathbb E}(\mathbbm{1}_{C}|{\mathcal K})$ 
    almost surely. Thus,
\begin{align*}
{\mathbb E}(\mathbbm{1}_A \mathbbm{1}_B\mathbbm{1}_C|{\mathcal K}) = 
    {\mathbb E}(\mathbbm{1}_A \mathbbm{1}_{B \cap C}|{\mathcal K}) 
    &= 
    {\mathbb E}(\mathbbm{1}_A|{\mathcal K}) {\mathbb E}(\mathbbm{1}_{B \cap C}|{\mathcal K}) \\
    &= {\mathbb E}(\mathbbm{1}_A|{\mathcal K}) {\mathbb E}(\mathbbm{1}_{B}\mathbbm{1}_C|{\mathcal K})
    ={\mathbb E}(\mathbbm{1}_A|{\mathcal K}) {\mathbb E}(\mathbbm{1}_{B}|{\mathcal K}) 
    {\mathbb E}(\mathbbm{1}_{C}|{\mathcal K})
    \end{align*}
    almost surely, which proves {\em (i)}.

    Consider now {\em (ii)}. Define $P = {\mathbb E}(\mathbbm{1}_A\mathbbm{1}_B|{\mathcal H})$ and $Q = {\mathbb E}(\mathbbm{1}_A|{\mathcal H}){\mathbb E}(\mathbbm{1}_B|{\mathcal H})$. 
   As $P$ and $Q$ are ${\mathcal H}$-measurable and $C \in {\mathcal H}$ is arbitrary, if ${\mathbb E}(P\mathbbm{1}_C) = {\mathbb E}(Q\mathbbm{1}_C)$ then 
   we can use uniqueness of conditional expectation to conclude $P = Q$ almost surely, 
   so that {\em(ii)} holds. Since ${\mathcal H} \subseteq {\mathcal G}$,  $B \cap C \in {\mathcal G}$.
    Moreover $A \in {\mathcal F}$ 
    and ${\mathcal F},{\mathcal G}$ 
    are independent, so
    \begin{equation}\label{condexp1}
        {\mathbb E}(P\mathbbm{1}_C) = {\mathbb E}(\mathbbm{1}_A \mathbbm{1}_B\mathbbm{1}_C) ={\mathbb P}(A \cap B \cap C) = {\mathbb P}(A){\mathbb P}(B \cap C).
    \end{equation}
    The first equality in~\eqref{condexp1} comes from definition of conditional expectation. Since $A \in {\mathcal F}$ 
    and ${\mathcal F},{\mathcal G}$ are independent, ${\mathbb E}(\mathbbm{1}_A|{\mathcal G}) = {\mathbb E}(\mathbbm{1}_A)$. 
    So by the tower property,
    \begin{equation*}
        {\mathbb E}(\mathbbm{1}_A|{\mathcal H})
        = {\mathbb E}({\mathbb E}(\mathbbm{1}_A|{\mathcal G})|{\mathcal H}) = 
        {\mathbb E}({\mathbb E}(\mathbbm{1}_A)|{\mathcal H})
        ={\mathbb E}(\mathbbm{1}_A)=
        {\mathbb P}(A).
    \end{equation*}
    Moreover since $\mathbbm{1}_C$ is ${\mathcal H}$-measurable, ${\mathbb E}(\mathbbm{1}_B|{\mathcal H})\mathbbm{1}_C = {\mathbb E}(\mathbbm{1}_B\mathbbm{1}_C|{\mathcal H})$. Thus,
    \begin{align}\begin{split}\label{condexp2}
        {\mathbb E}(Q\mathbbm{1}_C) &= {\mathbb E}({\mathbb E}(\mathbbm{1}_A|{\mathcal H})
        {\mathbb E}(\mathbbm{1}_B|{\mathcal H})\mathbbm{1}_C)\\
        &={\mathbb E}({\mathbb P}(A)
        {\mathbb E}(\mathbbm{1}_B\mathbbm{1}_C|{\mathcal H})) \\
        &= {\mathbb P}(A){\mathbb E}({\mathbb E}(\mathbbm{1}_{B\cap C}|{\mathcal H}))\\
        &= {\mathbb P}(A){\mathbb E}(\mathbbm{1}_{B\cap C}) = {\mathbb P}(A){\mathbb P}(B \cap C),\end{split}
    \end{align}
    with the last line using the tower property.
    Now {\em(ii)} follows from~\eqref{condexp1}-\eqref{condexp2}.
\end{proof}

\begin{proof}[Proof of Proposition~\ref{prop_synch}]
Adopt the notation of Assumption~\ref{A1}. 
Below let $j,k,\ell$ denote positive integers. It is easy to check that 
$r_{j} = r_k$ if and only if $k-j$ is an integer 
multiple of $R$, while $m_{k-\ell R} = m_k-\ell$ when 
$\ell \le m_k$. Thus, 
\begin{align}\begin{split}\label{ii-iii}
\{m_j: j <k,\,r_j = r_k\} &= 
\{m_{k-\ell R}: \ell \le  m_k\} \\
&= \{0,1,\ldots,m_k-1\},
\end{split}
\end{align}
with both sides empty if $m_k = 0$.
See Figure~\ref{fig_synchronous}. 
Define $$\zeta^r = \inf\{t\ge 0:Y^r(t) \notin U\}.$$
Fix $k \ge 2$ and note that, by definition of the fragments,
\begin{equation}\label{subset_of}
  \{\zeta^{r_j}>(m_{j}+1)\Delta t,\,\,\,\forall\,j <k\} \subseteq  \{T_{k-1}>\Delta t,\ldots,T_1>\Delta t\},
\end{equation}
where $E \subseteq E'$ indicates the event $E'$ occurs whenever $E$ occurs. By~\eqref{ii-iii}, 
\begin{align}\begin{split}\label{subsetof}
    &\{Y^{r_j}(t) \in U \,\,\,\forall\,t \in [m_j\Delta t,(m_j+1)\Delta t],j\le k\} \\
    &\subseteq \{Y^{r_k}(t) \in U \,\,\,\forall\, t \in [0,(m_k+1)\Delta t]\}.
    \end{split}
\end{align}
By definition of the fragments and~\eqref{subsetof},
\begin{align}\begin{split}\label{supset_of}
\{T_{k-1}>\Delta t,\ldots,T_1>\Delta t\} &=  \{Y^{r_j}(t) \in U, \,\,\,\forall\,t \in [m_j\Delta t,(m_j+1)\Delta t],j<k\} \\
&\subseteq \{Y^{r_j}(t) \in U, \,\,\,\forall\,t \in [0,(m_j+1)\Delta t],j<k\}
\\
&= \{\zeta^{r_j}>(m_{j}+1)\Delta t,\,\,\,\forall\,j <k\}.
\end{split}
\end{align}
Combining~\eqref{subset_of} and~\eqref{supset_of}, 
\begin{equation}\label{frag}
\{T_{k-1}>\Delta t,\ldots,T_1>\Delta t\} = \{\zeta^{r_j}>(m_{j}+1)\Delta t,\,\,\,\forall\,j<k\}.
\end{equation}
Due to 
independence of $Y^r(t)_{t \ge 0}$ over $r$ and Lemma~\ref{lem_basic}{\em (ii)},
\begin{align}\begin{split}\label{conditional_on}
    &\text{conditional on } \{\zeta^{r_j} > (m_j + 1)\Delta t,\,\,\,\forall\,j<k \,\,\,s.t.\,\,\, r_j = r_k\},\\
    &\text{the event }\{\zeta^{r_j} > (m_j + 1)\Delta t ,\,\,\,\forall\,j<k\,\,\,s.t.\,\,\,r_j \ne r_k\}\\
    &\text{is independent 
    of }Y^{r_k}(m_k\Delta t).
    \end{split}
\end{align}
Again using~\eqref{ii-iii},
\begin{equation}\label{zeta_rk}
\{\zeta^{r_k} > m_k\Delta t\} = \{\xi^{r_j}>(m_j+1)\Delta t, \,\,\,\forall j <k \,\,\,s.t.\,\,\,r_j = r_k\}.
\end{equation}
Combining~\eqref{frag},~\eqref{conditional_on}, and~\eqref{zeta_rk}, and using~\eqref{QSDdef},
\begin{align}\begin{split}\label{conclude}
&{\mathcal L}(X_{k}(0)|T_{k-1}>\Delta t,\ldots,T_1>\Delta t) \\
&= {\mathcal L}(Y^{r_{k}}(m_k\Delta t)|\zeta^{r_j}>(m_{j}+1)\Delta t,\,\,\,\forall\,j<k)\\
&= {\mathcal L}\left(\left.Y^{r_{k}}(m_k\Delta t)\right|\zeta^{r_j} > (m_j + 1)\Delta t,\,\,\,\,\forall\,j<k \,\,\,s.t.\,\,\,r_j = r_k\right)\\
&= {\mathcal L}\left(\left.Y^{r_{k}}(m_k\Delta t)\right|\zeta^{r_k}>m_k\Delta t\right)= \rho.
\end{split}
\end{align}
As $Y^1(t)_{t \ge 0}$
is a copy of $X(t)_{t \ge 0}$ 
with ${\mathcal L}(X(0))= \rho$, in particular 
${\mathcal L}(Y^1(0)) = \rho$. Thus, 
$${\mathcal L}(X_1(0)) = {\mathcal L}(Y^{r_1}(m_1)) = {\mathcal L}(Y^1(0)) = \rho.$$
We have now established~\eqref{assump2} of 
Assumption~\ref{A1}. 

Consider now~\eqref{assump1}. Let $k \ge 1$.
Due to independence of $Y^r(t)_{t \ge 0}$ over $r$ and Lemma~\ref{lem_basic}{\em(ii)}, 
we see that 
conditional on $X_k(0)$,  
$(X_\ell(t)_{0 \le t \le \Delta t})_{r_\ell = r_k}$ is independent 
of $(X_\ell(t)_{0 \le t \le \Delta t})_{r_\ell \ne r_k}$. 
For $k\ge 2$, the Markov property of 
$Y^{r_k}(t)_{t \ge 0}$ 
and~\eqref{ii-iii}
show that,
conditional on $X_k(0)$, 
$X_k(t)_{0 \le t \le \Delta t}$ is 
independent of $(X_\ell(t)_{0 \le t \le \Delta t})_{\ell < k, r_\ell = r_k}$. By Lemma~\ref{lem_basic}{\em(i)} with ${\mathcal F} = \sigma((X_\ell(t)_{0 \le t \le \Delta t})_{\ell<k, r_\ell\ne r_k})$, 
${\mathcal G} =\sigma((X_\ell(t)_{0 \le t \le \Delta t})_{\ell<k, r_\ell= r_k})$, 
${\mathcal H} = \sigma(X_k(t)_{0 \le t \le \Delta t})$ and ${\mathcal K} = \sigma(X_k(0))$, we have,
for $k\ge 2$,
\begin{equation}\label{I4}
\text{conditional on }X_k(0), 
\,X_k(t)_{0 \le t \le \Delta t} \text{ is 
independent of }(X_\ell(t)_{0 \le t \le \Delta t})_{\ell < k}.
\end{equation}

Now define the fragments' irrelevant 
futures as follows. Let $X_k(t)_{t \ge \Delta t}$ be 
copies of $X(t)_{t \ge 0}$ that evolve 
forward of time independently 
of everything else. That is, for each 
$k \ge 1$, conditional on $X_k(\Delta t)$, 
$X_k(t)_{t \ge \Delta t}$, $X_k(t)_{0 \le t \le 
\Delta t}$, and $(X_\ell(t)_{t \ge 0})_{\ell<k}$ 
are mutually independent. From~\eqref{I4} 
it is easy to see this is possible, 
as the irrelevant futures have no 
bearing on the definitions of the 
fragments. Now by construction of 
the irrelevant futures and~\eqref{I4}, 
it is easy to see that for $k\ge 2$,
conditional on $X_k(0)$, 
$X_k(t)_{t \ge 0}$ is 
independent of $(X_\ell(t)_{t \ge 0})_{\ell<k}$. 
This proves~\eqref{assump1} 
in Assumption~\ref{A1}.
\end{proof}

\begin{proof}[Proof of Proposition~\ref{prop_asynch}]
Adopt the notation of Assumption~\ref{A1}. 
This proof will follow 
the same basic steps as 
the proof of Proposition~\ref{prop_synch}, 
but the justifications will be 
different. Let $
\zeta^r = \inf\{t\ge 0:Y^r(t) \notin U\}$ be as above. 

Fix $k \ge 2$. We first 
claim that~\eqref{ii-iii} still 
holds. Let 
$n \in \{0,1,\ldots,m_k-1\}$. 
By the surjectivity assumption 
in {\em (iii)} there is 
$j$ such that $r_j = r_k$ 
and $m_j = n$. Since $m_j = n < m_k$ 
and $r_j = r_k$, from {\em (ii)} 
we have $t_{wall}^{r_j}(m_j) 
\le t_{wall}^{r_k}(m_k)$. Since 
$j \ne k$, using monotonicity 
in {\em (iii)} we conclude 
$j <k$. 
Thus $\{0,1,\ldots,m_k-1\} 
\subseteq \{m_j:j <k,\,r_j = r_k\}$. 
Now consider $m_j$ such 
that $j <k$ and 
$r_j = r_k$. By monotonicity 
in {\em (iii)} we must have 
$t_{wall}^{r_j}(m_j)<t_{wall}^{r_k}(m_k)$. Then by {\em (ii)} we can conclude 
$m_j < m_k$. Thus 
$\{m_j:j <k,\,r_j = r_k\} 
= \{0,1,\ldots,m_k-1\}$.

Next we establish~\eqref{assump2}.
Equipped with~\eqref{ii-iii}, 
we see that~\eqref{zeta_rk} holds. 
Moreover, since~\eqref{frag1} 
agrees with~\eqref{frag2}, 
the same steps as 
in the Proof of Proposition~\ref{prop_synch} show 
that~\eqref{frag} holds. 
On the other 
hand,~\eqref{conditional_on} 
holds because of {\em (i)}, Lemma~\ref{lem_basic}{\em(ii)},  
and independence
of $Y^r(t)_{t \ge 0}$ over $r$. 
The sequence of equalities 
in~\eqref{conclude} 
then holds, with the 
last equality using 
{\em (i)} again. It 
remains to show that ${\mathcal L}(X_1(0)) = 
\rho$. Suppose $m_1 > 0$. By surjectivity in 
{\em (iii)} there is $j>1$ such 
that $m_j = 0$ and $r_j = r_1$. 
But then {\em (ii)} implies $t_{wall}^{r_j}(m_j) \le t_{wall}^{r_1}(m_1)$, which 
contradicts monotonicity in 
{\em (iii)}. Thus $m_1 = 0$, so we can apply {\em (i)} 
to conclude ${\mathcal L}(X_1(0)) = {\mathcal L}(Y^{r_1}(m_1)) = \rho$.  
Thus~\eqref{assump2} in in Assumption~\ref{A1} holds.

Consider now~\eqref{assump1}. 
By {\em (i)} and independence of $Y^r(t)_{t \ge 0}$ over $r$, 
conditional on $X_k(0)$, 
$(X_\ell(t)_{0 \le t \le \Delta t})_{r_\ell = r_k}$ is independent 
of $(X_\ell(t)_{0 \le t \le \Delta t})_{r_\ell \ne r_k}$. 
Recall that~\eqref{ii-iii} still holds. Thus for $k\ge 2$, by the Markov property of 
$Y^{r_k}(t)_{t \ge 0}$ and~\eqref{ii-iii}, 
conditional on $X_k(0)$, 
$X_k(t)_{0 \le t \le \Delta t}$ is 
independent of $(X_\ell(t)_{0 \le t \le \Delta t})_{\ell < k, r_\ell = r_k}$. 
By Lemma~\ref{lem_basic}{\em (i)} 
we conclude~\eqref{I4} holds for $k \ge 2$. 
Let the trajectory fragments' irrelevant futures 
be independent of everything else 
as in the proof of Proposition~\ref{prop_synch}. 
Following the reasoning in that proof
we see that~\eqref{assump1} 
in Assumption~\ref{A1} holds.
\end{proof}

\begin{proof}[Proof of Theorem~\ref{thm_consistent}]
The statements {\em (i)} 
and {\em (ii)} follow from Propositions~\ref{prop_synch} 
and~\ref{prop_asynch}, 
respectively, with
$((\xi_n^r,\theta_n^r)_{n \ge 0})^{r=1,\ldots,R}$ 
taking the place of 
$(Y^r(t)_{t \ge 0})^{r=1,\ldots,R}$, 
and
with $t_m \equiv \Delta t = 1$ and $g(\xi,\theta) = \int_0^{\theta} f(\psi(t,\xi))\,dt$.
\end{proof}

\begin{proof}[Proof 
of Theorem~\ref{thm_consistent2}]
The statements {\em (i)} and {\em (ii)}  follow from Propositions~\ref{prop_synch} 
and~\ref{prop_asynch}, 
respectively, with $(Z^r(t)_{t \ge 0})^{r=1,\ldots,R}$ 
taking the place of 
$(Y^r(t)_{t \ge 0})^{r=1,\ldots,R}$, 
and with $t_m \equiv \Delta t$ and $g = f$.
\end{proof}

\subsection{Supplementary results}\label{sec:proofs2}

We first show that the decoupling 
of wall-clock times from 
the speed of computing 
$X(t)_{t \ge 0}$ is a necessary 
condition for consistency. 
Below we
break assumption 
{\em (i)} in Proposition~\ref{prop_asynch} by assuming the 
wall-clock times to 
obtain the
initial QSD samples 
$Y^r(0)$, $r=1,\ldots,R$,
in the parallel step 
are correlated with 
the positions of those 
samples.
\begin{remark}\label{rmk_bias}
In Proposition~\ref{prop_asynch}, 
if {(i)} does 
not hold,
then the conclusions 
of Theorem~\ref{theorem_dyn} 
and Theorem~\ref{theorem_avg} 
may not hold. 
\end{remark}
\begin{proof}[Example]
Let $X(t)_{t \ge 0}$ be a simple 
random walk on ${\mathbb Z}$, 
meaning $X(t+1)-X(t) = 1$ or $-1$, 
each with probability $1/2$. 
Let $U = \{0,1\}$. 
The QSD $\rho$ of $X(t)_{t \ge 0}$ in 
$U$ is simply the 
uniform distribution on $U$. Assume 
$X(t)_{t \ge 0}$ has initial 
distribution ${\mathcal L}(X(0)) = \rho$, 
and let $(Y^r(t)_{t \ge 0})^{r=1,\ldots,R}$ 
be independent copies of $X(t)_{t \ge 0}$. Suppose 
\begin{equation}\label{break}
    t_{wall}^r(0) < t_{wall}^s(0) 
    \qquad \text{whenever }
    Y^r(0) = 0 \text{  
and }Y^s(0) = 1.
\end{equation}
Notice that~\eqref{break} violates 
{\em (i)} of Proposition~\ref{prop_asynch}. 
Assume however that {\em (ii)}
and {\em (iii)} in Proposition~\ref{prop_asynch} hold. 
Then arguments similar 
to those in the proof 
of Proposition~\ref{prop_asynch} 
show that $\{Y^{r_1}(m_1) = 1\} 
= 
\{Y^r(0) = 1\,\,\,\forall\,r\}$. 
Adopt the notation of 
Algorithm~\ref{alg_gen}. 
Then by the above and the definition~\eqref{frag2} of the fragments, 
\begin{align}\begin{split}\label{agree1}
    {\mathbb P}(T_{par} = 1,X_{par} = 2) &= {\mathbb P}(Y^{r_1}(m_1) = 1,Y^{r_1}(m_1+1) = 2) \\
    &= \frac{1}{2}{\mathbb P}(Y^{r_1}(m_1) = 1)\\
    &= \frac{1}{2}{\mathbb P}(Y^r(0) = 1\,\,\,\forall\,r) = \frac{1}{2}\left(\frac{1}{2}\right)^R.
    \end{split}
\end{align}
Similarly,
\begin{align}\begin{split}\label{agree2}
    {\mathbb P}(T = 1,X(T)=2) = 
    {\mathbb P}(X(0) =1, X(1) = 2) 
    = \frac{1}{2}{\mathbb P}(X(0)=1) =\frac{1}{4}.
    \end{split}
\end{align}
Notice when $R>1$,~\eqref{agree1} and~\eqref{agree2} show  the conclusion of Theorem~\ref{theorem_dyn} does not 
hold, as
${\mathbb P}(T_{par} = 1,X_{par} = 2) 
\ne {\mathbb P}(T = 1,X(T)=2)$.
A similar construction shows 
the conclusion of Theorem~\ref{theorem_avg} can fail 
when {\em (i)} does not hold.
\end{proof}

The next two results below are formal 
calculations related to claims 
made in the text above. 
These results could be 
made precise using results 
in~\cite{durmus2018geometric,durmus2018piecewise}. 
However we stick to formal 
computations for brevity.

\begin{remark}\label{prop_balance}
Suppose~\eqref{ratebalance} holds for all 
 $x \in {\mathbb R}^{d-1}$ and $i \in {\mathcal I}$. Then $e^{-V(x)}$ 
is formally invariant for a PDMP generated by~\eqref{gen}.
\end{remark} 
\begin{proof}[Formal proof]
Let $L$ be defined as in~\eqref{gen}.
We will show that $$\sum_{i\in {\mathcal I}} \int_{{\mathbb R}^{d-1}} Lf(x,i) \pi(x,i)\,dx = 0$$ 
provided~\eqref{ratebalance} holds and
$\pi(x,i)\propto e^{-V(x)}$, where
 $\propto$ 
indicates proportional to.
Write $$\lambda_i(x,i) = -\sum_{j\ne i}\lambda_j(x,i).$$
With sufficient 
regularity we can integrate by parts to get
\begin{align}\begin{split}\label{formal1}
\sum_i \int Lf(x,i) \pi(x,i)\,dx &\propto
\sum_i \int \left(d_i \cdot \nabla f(x,i) + \sum_j \lambda_j(x,i)f(x,j)\right)e^{- V(x)}\,dx\\
&= -\sum_i \int f(x,i) \nabla \cdot \left(d_i e^{- V(x)}\right)dx \\
&\qquad \quad
+ \sum_i \sum_j \int \lambda_i(x,j)f(x,i)e^{- V(x)}\,dx \\
&= \sum_i \int f(x,i)\left(d_i \cdot \nabla V(x)  +\sum_j \lambda_i(x,j)\right)e^{-  V(x)}\,dx.
\end{split}
\end{align}
Above, all the sums are over ${\mathcal I}$ and integrals 
are over ${\mathbb R}^{d-1}$. If~\eqref{ratebalance} holds, 
\begin{align*}
  d_i \cdot \nabla V(x)  +\sum_j \lambda_i(x,j) &= 
  d_i \cdot \nabla V(x)  +\sum_{j\ne i} (\lambda_i(x,j)-\lambda_j(x,i)) = 0.
\end{align*}
Comparing with~\eqref{formal1} gives the result.
\end{proof}
Note that the calculation in 
Remark~\ref{prop_balance} 
shows~\eqref{ratebalance} is a necessary 
condition for~\eqref{gen} to define a 
PDMP with an invariant distribution of the form $\pi(x,i)\propto e^{-V(x)}$.

\begin{remark}\label{prop_balance1}
$e^{-\beta V(x)}$ 
is formally invariant for a PDMP 
generated by~\eqref{ell}. 
\end{remark}
\begin{proof}[Formal proof]
Let $L$ be defined as in~\eqref{ell}. We will show that $$\sum_{k=0}^{N-1} \int_\Omega Lf(x,k) \pi(x,k)\,dx = 0$$ provided
$\pi(x,k) \propto e^{-\beta V(x)}$. Recall $d_0,\ldots,d_{N-1} \in {\mathbb R}^{d-1}$ sum to $0$ and we
consider the
indices of the $d_k$'s as elements of ${\mathbb Z}_N$, 
the integers modulo $N$. Write 
$$F_{k,\ell}(x) = \beta(d_k+\ldots+d_{k+\ell})\cdot \nabla V(x).$$
With 
sufficient regularity we can integrate by parts to get
\begin{align*}
&\sum_{k =0}^{N-1} \int_\Omega Lg(x,k){\pi}(x,k)\,dx \\
&\propto 
\sum_{k =0}^{N-1} \int_\Omega \left(d_k \cdot \nabla g(x,k) + [g(x,k-1)-g(x,k)]\max_{0 \le \ell \le N-1}F_{k,\ell}(x)\right)e^{-\beta V(x)}\,dx \\
&= -\sum_{k =0}^{N-1} \int_\Omega g(x,k)\nabla \cdot \left(d_k e^{-\beta V(x)}\right)dx\\
&\qquad \quad+ \sum_{k =0}^{N-1} \int_\Omega g(x,k) \left(\max_{0 \le \ell \le N-1} F_{k+1,\ell}(x) - \max_{0 \le \ell \le N-1}F_{k,\ell}(x)\right)e^{-\beta V(x)}\,dx\\
&=\sum_{k =0}^{N-1} \int_\Omega  g(x,k)\left(\beta d_k \cdot \nabla V(x)+\max_{0 \le \ell \le N-1} F_{k+1,\ell}(x) -  \max_{0 \le \ell \le N-1}F_{k,\ell}(x) \right)e^{-\beta V(x)}\,dx,
\end{align*}
where when $\Omega = {\mathbb R}^{d-1}$ we assume $V$ grows 
sufficiently fast at $\infty$ 
so that we can 
neglect the boundary term 
from the integration by 
parts.
Observe that, because $\sum_{\ell=0}^{N-1} d_{k+\ell} = 0$ and 
$d_{k+N} = d_k$, we have
\begin{align*}
   & \{d_k+d_{k+1},d_k+d_{k+1}+d_{k+2},\ldots,d_k+\ldots+d_{k+N-1},d_k+\ldots+d_{k+N}\} \\
   & = \{d_k+d_{k+1},d_k+d_{k+1}+d_{k+2},\ldots,d_k+\ldots+d_{k+N-1},d_k\}\\
   &= \{d_k,d_{k}+d_{k+1},\ldots,d_{k+N-1}\}.
\end{align*}
It follows that
\begin{align*}
    &\beta d_k \cdot \nabla V(x)+\max_{0 \le \ell \le N-1} F_{k+1,\ell}(x) -  \max_{0 \le \ell \le N-1}F_{k,\ell}(x) \\
    &= \beta\max_{0 \le \ell \le N-1}
    (d_k+\ldots+d_{k+\ell+1})\cdot \nabla V(x) - \beta\max_{0 \le \ell \le N-1}(d_k+\ldots+d_{k+\ell})\cdot \nabla V(x)\\
    &= 0.
\end{align*}
This proves the desired result.
\end{proof}

\begin{remark}\label{rmk_ell}
$e^{-\beta V}$ is invariant for $Z(n\delta t)_{n \ge 0}$ 
defined in Algorithm~\ref{alg_ell}.
\end{remark}

\begin{proof}
Write the acceptance 
probability in Algorithm~\ref{alg_ell} as 
$$A_k(x) = \min_{0 \le \ell \le N-1}\exp\left(\beta V(x)-\beta V(x+d_k\delta t+\ldots + d_{k+\ell}\delta t)\right).$$
Arguing similarly as in the formal proof of Remark~\ref{prop_balance1}, since 
$\sum_{k=0}^{N-1} d_k = 0$
we have 
\begin{equation*}
    \frac{A_k(x)}{A_{k+1}(x+d_k\delta t)} = \exp(\beta V(x)-\beta V(x+d_k\delta t)).
\end{equation*}
Now let $\pi(x,k) \propto e^{-\beta V(x)}$. Then the last display shows that 
\begin{equation}\label{show_inv}
\pi(x+d_k\delta t,k) = \pi(x,k)A_k(x)+\pi(x+d_k\delta t,k+1)(1-A_{k+1}(x+d_k\delta t)).    
\end{equation}
Inspecting Algorithm~\ref{alg_ell}, 
we see that~\eqref{show_inv} 
demonstrates the required result.
\end{proof}

\section*{Acknowledgements}

The author gratefully 
thanks Peter Christman for 
producing the numerical 
results leading to Figures~\ref{fig1},~\ref{fig2} and~\ref{fig3} in Section~\ref{sec:numerics}, 
as well as 
Petr Plech{\'a}{\v c}, Gideon Simpson, and Ting Wang for helpful conversations. The author 
also gratefully acknowledges support 
from the
National Science Foundation 
via the awards
NSF-DMS-1522398 and 
NSF-DMS-1818726.

\bibliographystyle{siamplain}
\bibliography{bibliography.bib}

\end{document}